\documentclass[a4paper,11pt]{amsart}

\usepackage{comment}

\usepackage{etoolbox}
\ifdef{\crop}{%
\usepackage[includeheadfoot,twoside=False,paperwidth=448pt,paperheight=587pt,rmargin=15pt,lmargin=15pt,tmargin=15pt,bmargin=15pt]{geometry}%
}{%
\setlength{\topmargin}{22mm}
\addtolength{\topmargin}{-1in}
\setlength{\oddsidemargin}{27mm}
\addtolength{\oddsidemargin}{-1in}
\setlength{\evensidemargin}{27mm}
\addtolength{\evensidemargin}{-1in}
\setlength{\textwidth}{156mm}
\setlength{\textheight}{240mm}
}%

\usepackage{color}
\usepackage[active]{srcltx}
\usepackage{amsmath,amsthm,amsxtra}
\usepackage{amssymb}

\usepackage[mathscr]{eucal}

\usepackage{bm}
\usepackage[all]{xy}

\usepackage{aliascnt}

\theoremstyle{plain}
\newtheorem{thm}{Theorem}[section]
\newtheorem*{thm*}{Theorem}

\newaliascnt{prop}{thm}
\newaliascnt{cor}{thm}
\newaliascnt{lem}{thm}
\newaliascnt{claim}{thm}
\newaliascnt{defn}{thm}
\newaliascnt{ques}{thm}
\newaliascnt{conj}{thm}
\newaliascnt{fact}{thm}
\newaliascnt{rem}{thm}
\newaliascnt{ex}{thm}
\newtheorem{prop}[prop]{Proposition}
\newtheorem{cor}[cor]{Corollary}
\newtheorem{lem}[lem]{Lemma}
\newtheorem{claim}[claim]{Claim}
\newtheorem*{prop*}{Proposition}
\newtheorem*{cor*}{Corollary}
\newtheorem*{lem*}{Lemma}
\newtheorem*{claim*}{Claim}
\theoremstyle{definition}
\newtheorem{defn}[defn]{Definition}
\newtheorem{ques}[ques]{Question}

\newtheorem*{defn*}{Definition}
\newtheorem*{ques*}{Question}
\newtheorem*{conj*}{Conjecture}

\newtheorem*{prob*}{Problem}

\newtheorem{rem}[rem]{Remark}
\newtheorem{ex}[ex]{Example}
\newtheorem*{fact*}{Fact}
\newtheorem*{rem*}{Remark}
\newtheorem*{ex*}{Example}
\aliascntresetthe{prop}
\aliascntresetthe{cor}
\aliascntresetthe{lem}
\aliascntresetthe{claim}
\aliascntresetthe{defn}
\aliascntresetthe{ques}
\aliascntresetthe{conj}
\aliascntresetthe{fact}
\aliascntresetthe{rem}
\aliascntresetthe{ex}

\usepackage[pointedenum]{paralist}
\usepackage{varioref}
\labelformat{equation}{\textnormal{(#1)}}
\labelformat{enumi}{\textnormal{(#1)}}

\usepackage[a4paper]{hyperref}

\def\textsectionN~{\textsection{}}

\renewcommand\phi{\varphi}
\renewcommand\epsilon{\varepsilon}
\renewcommand\leq{\leqslant}
\renewcommand\geq{\geqslant}

\makeatletter
\newcommand{\set}{  \@ifstar{\@setstar}{\@set}}\newcommand{\@setstar}[2]{\{\, #1 \mid #2 \,\}}
\newcommand{\@set}[1]{\{ #1 \}}
\makeatother

\newcommand{\trans}[1][1]{\raisebox{#1ex}{\scriptsize\kern0.1em$t$\kern-0.1em}}

\DeclareMathOperator{\mult}{mult}
\DeclareMathOperator{\Pic}{Pic}

\DeclareMathOperator{\Supp}{Supp}

\DeclareMathOperator{\vol}{vol}

\DeclareMathOperator{\lct}{lct}

\DeclareMathOperator{\Nklt}{Nklt}

\def\N{\mathbb{N}}
\def\Z{\mathbb{Z}}
\def\Q{\mathbb{Q}}
\def\R{\mathbb{R}}
\def\C{\mathbb{C}}

\def\r+{\mathbb{R}_{\geq 0}}

\def\ep{\varepsilon}

\def\r+{{\R}_{\geq 0}}
\def\q+{{\Q}_{\geq 0}}
\def\P{\mathbb{P}}
\def\arw{\rightarrow}
\def\*c{\C^{\times}}

\def\C{\mathbb {C}}

\def\N{\mathbb {N}}

\def\Q{\mathbb {Q}}
\def\R{\mathbb {R}}

\def\Z{\mathbb {Z}}

\newcommand{\calf}{\mathcal {F}}

\newcommand{\cali}{\mathcal {I}}
\newcommand{\calj}{\mathcal {J}}

\newcommand{\calo}{\mathcal {O}}

\makeatletter
  
  \@addtoreset{equation}{section}
\makeatother

\title[Bpf thresholds and higher syzygies on abelian threefolds]{Basepoint-freeness thresholds and higher syzygies on abelian threefolds}

\author[A.~Ito]{Atsushi~Ito}
\address{Graduate School of Mathematics,
Nagoya University,
Nagoya, Japan}
\email{atsushi.ito@math.nagoya-u.ac.jp}


\begin{document}

\maketitle

\begin{abstract}
For a polarized abelian variety,
Z.\ Jiang and G.\ Pareschi introduce an invariant and show that the polarization is basepoint free or projectively normal 
if the invariant is small.
Their result is generalized to higher syzygies by F.\ Caucci, that is, the polarization 
satisfies property $(N_p)$ if the invariant is small.

In this paper, we study a relation between the invariant and degrees of abelian subvarieties with respect to the polarization.
For abelian threefolds, we give an upper bound of the invariant using degrees of abelian subvarieties.
In particular,
we affirmatively answer some questions on abelian varieties asked by the author, V.~Lozovanu and Caucci in the three dimensional case.
\end{abstract}

\section{Introduction}\label{sec_intro}

Throughout this paper, we work over the complex number field $\C$.
In \cite{MR4157109}, Z.\ Jiang and G.\ Pareschi introduce \emph{cohomological rank functions} of
$\Q$-twisted (complexes of) coherent sheaves on polarized abelian varieties.
Using cohomological rank functions,
they define the \emph{basepoint-freeness threshold} $\beta(l)$ for a polarized abelian variety $(X,l)$ and show the following:
\begin{enumerate}
\item $0 < \beta(l)  \leq 1$, and $\beta(l)  < 1$ if and only if
any line bundle $L$ representing $l$ is basepoint free.
\item If $\beta(l)  < \frac1{2}$,
any line bundle $L$ representing $l$ is projectively normal.
\end{enumerate}
F.~Caucci  generalizes  (2) to higher syzygies,
proving that
any line bundle $L$ representing $l$ satisfies property $(N_p)$ if $\beta(l) < 1/(p+2)$ in \cite{MR4114062}\footnote{Caucci proves the result on an algebraically closed field of any characteristic.}.
We refer the readers to 
\cite[Chapter 1.8.D]{MR2095471}, \cite{MR2103875} for the definition of  $(N_p)$.
We just note here that $(N_p)$'s consist an increasing sequence of positivity properties.
For example, 
$L$ satisfies ($N_0$) if and only if $L$ defines a projectively normal embedding,
and $L$ satisfies ($N_1$) if and only if $L$ satisfies ($N_0$) and the homogeneous ideal of the embedding is generated by quadrics. 

In \cite[Theorem 1.1]{MR4114062},
Caucci asks the following question:

\begin{ques}[\cite{MR4114062}]\label{ques2}
Let $(X,l)$ be a polarized abelian variety.
Set $B_t=t l$ for a positive rational number $t > 0$.
Assume that $(B_t^{\dim Z} .Z) > (\dim Z)^{{\dim Z}}$ for any abelian subvariety $Z \subset X$ with $\dim Z \geq 1$.
Then does it hold that  $\beta(l) < t$?
\end{ques}


For $\dim X \leq 2$,
\autoref{ques2} is answered affirmatively; 
this can be checked easily for elliptic curves. For abelian surfaces,
it is deduced from \cite[Proposition 3.1]{MR3923760} and \cite[Proposition 1.4]{MR4114062}
as explained in \cite{{MR4114062}}.
For $\dim X=3$, the proof of 
\cite[Theorem 1.1]{lozovanu2018singular} 
with \cite[Proposition 1.4]{MR4114062} gives a partial affirmative answer,
i.e.\ 
$\beta(l) < t$ if $ (B_t^3) > 59, (B_t^2.S) >4, (B_t.C) >2$ for any abelian surface $S \subset X$ and any elliptic curve $C \subset X$.
In any dimension,
it is not difficult to show $\beta(l) < t $ under a stronger assumption $(B_t^{\dim Z} .Z) > (\dim X \cdot \dim Z)^{{\dim Z}}$ 
by combining results in \cite{MR2833789} and \cite{MR1393263}
(see \autoref{prop_bounds_any_dim}).

In this paper,
we give an affirmative answer to \autoref{ques2} in the three dimensional case:

\begin{thm}\label{main_thm}
Let $(X,l)$ be a polarized abelian threefold.
Set $B_t=tl$ for a positive rational number $t >0$.
Assume that $ (B_t^3) > 27, (B_t^2.S) >4, (B_t.C) >1$ for any abelian surface $S \subset X$ and any elliptic curve $C \subset X$.
Then $\beta(l) <t$.
\end{thm}


We obtain the following corollary,
which answers \cite[Question 4.2]{MR3923760} and  \cite[Conjecture 7.1]{lozovanu2018singular} affirmatively in the three dimensional case:

\begin{cor}\label{cor_main}
Let $L$ be an ample line bundle on an abelian threefold $X$
and $p \geq 0$ be an integer.
Assume that $ (L^3) > 27(p+2)^3, (L^2.S) >4(p+2)^2, (L.C) > p+2$ for any abelian surface $S \subset X$ and any elliptic curve $C \subset X$.
Then 
$L$ satisfies $(N_p)$ and the Koszul cohomology group $K_{p,q}(X,L;dL) =0$ for any $q,d \geq 1$.
\end{cor}

The basepoint freeness on abelian threefolds is studied in
\cite{MR1248757}, \cite{MR1679235}, \cite{MR1821241}, etc.
To the best of the author's  knowledge, 
the following criterion did not appear in the literature.

\begin{cor}\label{cor_ Bpf}
Let $L$ be an ample line bundle on an abelian threefold $X$ with $(L^3) \geq 30$,
or equivalently $h^0(L) \geq 5$.
Then $L$ is basepoint free if and only if 
$(L^2.S) \geq 6 $ and $ (L.C) \geq 2$ for any abelian surface $S \subset X$ and any elliptic curve $C \subset X$.
\end{cor}

\vspace{2mm}
This paper is organized as follows. 
In \autoref{sec_preliminary}, we prepare some notation and lemmas.
We also show two propositions on basepoint-freeness thresholds.
In \autoref{sec_by_seshadri_constant},
we see a relation between basepoint-freeness thresholds and Seshadri constants.
In particular, we give a partial answer to \autoref{ques2} in arbitrary dimensions
in \autoref{prop_bounds_any_dim}.
In \autoref{sec_surface},
we study basepoint-freeness thresholds on abelian surfaces.
In \autoref{sec_by_multiplier_ideal},
we give an easy but important observation.
In \autoref{section_induction},
we show \autoref{prop_conj1_ques3},
which reduces \autoref{ques2} to another question.
In \autoref{sec_curve_hypersurface},
we show two lemmas on curves and hypersurfaces in abelian varieties.
In \autoref{sec_proof},
we give an affirmative answer to the question  in \autoref{section_induction}  for $\dim X \leq 3$
and prove \autoref{main_thm}.

In the proof of \autoref{main_thm},
we do not need results in \autoref{sec_by_seshadri_constant}-\autoref{sec_by_multiplier_ideal} other than \autoref{prop_local_enough'}. 
Hence the readers may read only
\autoref{sec_preliminary}, \autoref{prop_local_enough'}, and \autoref{section_induction}-\autoref{sec_proof}
for the proof of \autoref{main_thm}.


\subsection*{Acknowledgments}
The author would like to express his gratitude to
Professor Yoshinori Gongyo for answering his questions about minimal lc centers.
He also thanks Professor Zhi Zhang for valuable comments. 
The author was supported by Grant-in-Aid 
for Scientific Research (17K14162).

\section{Preliminaries}\label{sec_preliminary}

\subsection{Basepoint-freeness threshold}
Let $X$ be an abelian variety of dimension $g$.
An \emph{abelian subvariety} $Y \subset X$ is a subvariety which is closed under the group low of $X$.
In particular, $o \in Y$ and $Y$ can be $\{o\}$ or $X$.
A \emph{polarization} $l$ of $X$ is an ample class in $ \Pic(X)/\Pic^0(X)$.
For $b \in \Z$, the multiplication-by-$b$ isogeny is denoted by
\[
\mu_b : X \rightarrow X , \quad p \mapsto bp.
\]
It is known that
$\mu_b^* l = b^2 l$ and $\deg \mu_b = b^{2g}$.
Frequently $L$ denotes an ample line bundle representing $l$.

In \cite{MR4157109}, Jiang and Pareschi introduce the \emph{cohomological rank function} as follows:

\begin{defn}[{\cite[Definition 2.1]{MR4157109}}] \label{def_coh_rank_function}
Let $(X,l)$ be a $g$-dimensional polarized abelian variety and $\calf \in \mathrm{D}^b(X)$ be a bounded complex of coherent sheaves on $X$.
\begin{enumerate}
\item For $i \in \Z$,
$h^i_{gen}(X,\calf) $ denotes 
the dimension of the hypercohomology $H^i(X, \calf \otimes P_{\alpha})$
for general $\alpha \in \widehat{X}=\Pic^0(X)$,
where $P_{\alpha} $ is the numerically trivial line bundle on $X$ corresponding to $\alpha$.
\item The \emph{cohomological rank function}
$h^i_{\calf,l} : \Q \rightarrow \Q_{\geq0}$
is defined as
\[
h^i_{\calf,l} (x) =h^i_{\calf}(xl) := b^{-2g} h^i_{gen}(X, \mu_b^* \calf \otimes L^{ab}) 
\]
for $x = \frac{a}{b} \in \Q, b > 0$ and an ample line bundle $L$ representing $l$.
\end{enumerate}
\end{defn}

\begin{defn}[{\cite[Section 8]{MR4157109}}]\label{def_ Bpf_threshold}
Let $(X,l)$ be a polarized abelian variety.
The \emph{basepoint-freeness threshold} $\beta(l)$, or \emph{Bpf threshold} for short, is defined by
\[
\beta(X,l)=\beta(l) :=\inf \{ x \in \Q \, | \,  h^1_{\cali_p}(xl) =0 \}
\]
for a closed point $p \in X$.
\end{defn}

We note that $h^i_{\calf,l} (x) $ does not depend on the choices of $L$ nor the representation $x = \frac{a}{b}$,
and $\beta(l)$ does not depend on the choice of $p \in X$.

\begin{thm}[{\cite[Theorem D, Corollary E]{MR4157109}, \cite[Theorem 1.1]{MR4114062}}]\label{thm_ Bpf_threshold}
Let $(X,l)$ be a polarized abelian variety and $p \geq 0$ be an integer. 
\begin{enumerate}
\item $0 < \beta(l)  \leq 1$, and $\beta(l)  < 1$ if and only if
any line bundle $L$ representing $l$ is basepoint free.
\item If $\beta(l)  < \frac1{p+2}$,
any line bundle $L$ representing $l$ satisfies property  $(N_p)$.
\end{enumerate}
\end{thm}

For a coherent sheaf $\calf$ on $X$ and $x \in \Q$,
a \emph{$\Q$-twisted coherent sheaf} $\calf \langle xl \rangle$ is the equivalence class of the pair $(\calf,xl)$,
where the equivalence is defined by
\[
(\calf \otimes L^{m} , xl) \sim (\calf  , (x+m)l)
\]
for any line bundle $L$ representing $l$ and any $m \in \Z$.
In \cite{MR4157109}, 
the usual notions of generic vanishing are extended to the $\Q$-twisted setting.
A coherent sheaf $\calf$ on $X$ is said to be \emph{IT(0)} if $h^i(X,\calf \otimes P_{\alpha}) =0$ for any $i >0$ and any $\alpha \in \widehat{X}$.
A $\Q$-twisted coherent sheaf $\calf \langle xl \rangle$ for $x=\frac{a}{b}$ is said to be  \emph{IT(0)}
if so is $\mu_b^* \calf \otimes L^{ab}$.
We do not state the usual definition of the notion \emph{GV} here,
but 
we can define 
$\calf \langle xl \rangle$ to be \emph{GV} if
$\calf \langle (x+x')l \rangle$ is IT(0) for any rational number $x' >0$ by \cite[Theorem 5.2]{MR4157109}.

For a line bundle $ N$ on $X$,
$N  \langle xl \rangle$ is IT(0) if and only if 
$N  \langle xl \rangle$ is ample,
i.e.\ $N+xl$ is ample.
For the non-twisted case,
this is shown by
\cite[Example 3.10 (1)]{MR2435838}.
The $\Q$-twisted case follows from the non-twisted case.
Hence $N  \langle xl \rangle$ is GV if and only if 
$N  \langle xl \rangle$ is nef,
i.e.\ $N+xl$ is nef.

These notions are closely related to
$\beta(l)$ as follows:

\begin{lem}[{\cite[Section 8]{MR4157109},\cite[Lemma 3.3]{MR4114062}}]\label{lem_beta_IT(0)}
Let $(X,l)$ be a polarized abelian variety and $x \in \Q$.
Then $\beta(l) < x$ if and only if $\cali_p \langle xl \rangle$ is IT(0) for some (and hence for any) point $p \in X$.
\end{lem}

\begin{rem}\label{rem_curve}
If $X$ is an elliptic curve,
it is easy to check that $h^1_{\cali_p}(xl) =0 $ if and only if $ x \deg(l) -1 \geq 0$ by the Riemann-Roch theorem.
Hence we have $\beta(l) = \deg(l)^{-1}$.
This follows from  \autoref{lem_beta_IT(0)} as well.
\end{rem}

For a basepoint free ample line bundle $N$ and a line bundle $N'$ on a projective variety $X$,
the \emph{Koszul cohomology group} $K_{p,q}(X,N';N) $ is defined to be the cohomology of the Koszul-type
complex
\[
\bigwedge^{p+1} H^0(N) \otimes H^0(N'\otimes N^{q-1}) \rightarrow \bigwedge^{p} H^0(N) \otimes H^0(N' \otimes N^q)
 \rightarrow  \bigwedge^{p-1} H^0(N) \otimes H^0(N' \otimes N^{q+1})
\]
for $p,q \geq 0$.
For $p\geq 0$,
it is known that
$N$ satisfies  $(N_p)$ if and only if 
$K_{i,q}(X,\calo_X;N)=0 $ for any $ 0 \leq i \leq p$ and $q \geq 2$.

Let $M_N$ be the kernel of the surjective map $H^0(X,N) \otimes \calo_X \rightarrow N$.
By \cite[Section 1]{MR1193597},
the vanishing $K_{p,q}(X,N';N) =0$ follows from
$h^1( X, \bigwedge^{p+1} M_{N} \otimes N^{q-1} \otimes N')=0 $.
The following proposition is a variant of  \cite[Proposition 3.5]{MR4114062}.

\begin{prop}\label{prop_vanishing_K}
Let $X$ be an abelian variety 
and $N, N'$ be line bundles on $X$ such that $N$ is ample and basepoint free.
Let $n \in \Pic X/ \Pic^0(X)$ be the class of $N$ and
$p ,q $ be nonnegative integers.
Assume that that there exists a rational number $t>0$ such that $(q-1-t)N+N'$ is nef and $\beta(n)  < \frac{t}{p+1+t}   $.
Then $h^1( X, \bigwedge^{p+1} M_{N} \otimes N^{q-1} \otimes N')=0 $ 
and hence $K_{p,q}(X,N';N) =0$. 

In particular,
if $l$ is a polarization of $X$ and
$\beta(l)  < \frac{1}{p+2} $ for $p \geq 0$,
then $L$ is projectively normal and 
$K_{p,q}(X,L;dL) =0$ for any $q, d \geq 1$ and any $L$ representing $l$.
\end{prop}

\begin{proof}
Set $ s(n) =\inf \{ x \in \Q \, | \, h^1_{M_N}(xl) =0 \}$ as in \cite{MR4157109}.
By \cite[Theorem D]{MR4157109} and $\beta(n)  < \frac{t}{p+1+t}   $, we have
\[
s(n) =\frac{\beta(n)}{1-\beta(n)} < \frac{t}{p+1}
\]
and hence 
$M_{N} \langle  \frac{t}{p+1} n \rangle$ is IT(0) by \cite[Section 8]{MR4157109}, \cite[Lemma 3.3]{MR4114062}.
As a $\Q$-twisted sheaf,
$M_{N}^{\otimes p+1} \otimes N^{q-1} \otimes N'$ is written as
\[
\left(M_{N} \left\langle  \frac{t}{ p+1} n \right\rangle \right)^{\otimes p+1} \otimes ( N^{q-1}  \otimes N') \langle - t n \rangle.
\]
Since $(q-1-t)N+N'$ is nef by assumption,  
$(N^{q-1} \otimes N') \langle - t n \rangle$ is GV. 
Thus $M_{N}^{\otimes p+1} \otimes N^{q-1} \otimes N'$ is IT(0) by \cite[Proposition 3.4]{MR4114062}, 
and hence $h^1(X,M_{N}^{\otimes p+1} \otimes N^{q-1} \otimes N' )=0$.
Since $ \bigwedge^{p+1} M_{N} \otimes N^{q-1} \otimes N' $ is a direct summand of $M_{N}^{\otimes p+1} \otimes N^{q-1} \otimes N'$,
we have $h^1( X, \bigwedge^{p+1} M_{N} \otimes N^{q-1} \otimes N')=0 $ and hence $K_{p,q}(X,N';N) =0$.

For the last statement, the projective normality of $L$ follows from $\beta(l) < 1/(p+2) \leq 1/2$ 
and \autoref{thm_ Bpf_threshold} (2).
Let $N=dL, N'=L$ and $t=d^{-1}$ for $d \geq 1$.
Then
\[
\beta(n)  =\beta (dl) = d^{-1}\beta (l)  < d^{-1} \frac{1}{p+2} \leq \frac{d^{-1}}{p+1+d^{-1}}  
\]
for any $d \geq 1$.
Since $(q-1 -d^{-1}) N +N'=(q-1)N$ is nef for $q \geq 1$ and $\beta(n)   <  \frac{d^{-1}}{p+1+d^{-1}} $,
$K_{p,q}(X,L;dL) =0$ follows.
\end{proof}

\begin{rem}
By \autoref{prop_vanishing_K},
we can generalize 
\cite[Theorems 2.4, 2.9]{MR3483065}
as follows:

Let $N, N'$ be line bundles on $X$ such that $N$ is ample and basepoint free
as in \autoref{prop_vanishing_K}.
Assume that $g=\dim X \geq 2$, $N-N'$ is ample
and set $b=\min\{ s \in \R \, | \,  sN-N' \text{ is nef}  \} <1$.
Then $K_{p,1}(X,N';N)=0$
for 
\begin{align}\label{eq_range_p}
r - b -\frac{g-1-b}{\beta(n)}  < p \leq r-g,
\end{align}
where $r=h^0(N)-1$.

In fact, by Serre duality and the Kodaira vanishing theorem, we have
  \begin{equation}\label{eq_h^{g-1}}
\begin{aligned}
h^1( \wedge^{p+1} M_{N} \otimes N') &=h^{g-1}( \wedge^{r-p-1} M_{N} \otimes N  \otimes N'^{-1}) \\
&=h^{1}( \wedge^{r-p-g+1} M_{N} \otimes N^{g-1}  \otimes N'^{-1}) 
\end{aligned}
  \end{equation}
as in the proof of \cite[Lemma 2.1]{MR3483065}.
Here we use $p \leq r-g$ and the ampleness of $N-N'$.

Set $p'=r-p-g\geq 0$.
By \ref{eq_h^{g-1}}, $K_{p,1}(X,N';N)=0$ follows from $ h^{1}( \wedge^{p'+1} M_{N} \otimes N^{g-1}  \otimes N'^{-1})=0$.
Hence it suffices to find $t>0$ such that $(g-1-t)N-N'$ is nef and $\beta(n) < \frac{t}{p'+1+t}$ by \autoref{prop_vanishing_K}. 

Note that $g-1-b$ is positive by $g \geq 2$ and $b <1$.
We show that a rational $t>0$ with   $ 0< (g-1-b) - t \ll 1 $ satisfies these conditions.

By the definition of $b$ and $g-1-t > b$, $(g-1-t)N -N'$ is ample.
On the other hand, the first inequality of \ref{eq_range_p} is equivalent to $ \beta(n) < \frac{g-1-b}{r-b-p} =\frac{g-1-b}{p'+1+(g-1-b)}$.
Since $t$ is sufficiently close to $g-1-b$, we have $\beta(n) < \frac{t}{p'+1+t}$.
Hence we obtain  the vanishing $K_{p,1}(X,N';N)=0$.
\end{rem}

The following lemma is well-known to experts, but we give a proof for completeness.

\begin{lem}\label{lem_IT(0)}
Let $(X,l)$ be a polarized abelian variety
and $\calf$ be a coherent sheaf on $X$.
Let $x \in \Q$.
\begin{enumerate}
\item Let $0 \rightarrow \calf_1 \rightarrow \calf \rightarrow \calf_2 \rightarrow 0$ be an exact sequence of coherent sheaves on $X$.
If $\calf_1 \langle xl \rangle  $ and $\calf_2 \langle xl \rangle$ are IT(0), so is $\calf \langle xl \rangle$.
\item Let $(X',l')$ be a polarized abelian variety and $\calf'$ be a coherent sheaf on $X'$.
If $\calf \langle xl \rangle $ and $\calf' \langle xl' \rangle$ are IT(0),
so is $ \calf \boxtimes \calf' \langle x (l \boxtimes l') \rangle$,
where $\calf \boxtimes \calf' $ (resp.\ $l \boxtimes l' $) is the tensor product of pullbacks of $\calf $ and $\calf'$ (resp.\ $l$ and $l'$) on $X \times X'$.
\item Let $f : X' \rightarrow X$ be an isogeny.
Then $\calf \langle xl \rangle$ is IT(0) if and only if  so is $f^* \calf \langle x f^* l \rangle$.
\item Let $X' \subset X$ be an abelian subvariety and $\calf'$ be a coherent sheaf on $X'$. 
Then $\calf' \langle xl|_{X'} \rangle$ is IT(0) on $X'$ if and only if so is $ \iota_* \calf' \langle x l\rangle$ on $X$,
where $\iota :X' \rightarrow X$ is the inclusion morphism.
\end{enumerate}
\end{lem}

\begin{proof}
Let $x =\frac{a}{b}$ and $\mu'_b$ be the multiplication-by-$b$ isogeny on $X'$.\\
(1) For $\alpha \in \widehat{X}$,
we have $h^i(\mu_b^* \calf_j \otimes L^{ab} \otimes P_{\alpha}) =0$ for $i >0$ if $\calf_j \langle xl \rangle$ is IT(0) for $j=1,2$.
Tensoring $L^{ab} \otimes P_{\alpha}$ with the pullback of the exact sequence in (1) by $\mu_b$,
it holds that  $h^i(\mu_b^* \calf \otimes L^{ab} \otimes P_{\alpha}) =0$ for $i >0$.
Hence $\calf \langle xl \rangle$ is IT(0).\\[1mm]
(2) 
Let $(\alpha,\alpha')  \in \widehat{X} \times \widehat{X}' =\widehat{X \times X'}$.
By the K\"unneth formula,
\[
H^i(X \times X' , (\mu_b \times \mu'_b)^* (\calf \boxtimes \calf') \otimes (L \boxtimes L')^{ab} \otimes (P_{\alpha} \boxtimes P_{\alpha'})) 
\]
is the direct sum of 
\begin{align}\label{eq_H^k}
 H^k( X, \mu_b^* \calf  \otimes L^{ab} \otimes  P_{\alpha} ) \otimes H^{i-k}( X', {\mu'_b}^* \calf'  \otimes L'^{ab} \otimes   P_{\alpha'} ) 
\end{align}
for $0 \leq k \leq i$.
If $i >0$,
\ref{eq_H^k} is zero for any $0 \leq k \leq i$ 
since $\calf \langle xl \rangle $, $\calf' \langle xl' \rangle$ are IT(0).
Hence $ \calf \boxtimes \calf' \langle x (l \boxtimes l') \rangle$ is IT(0) as well.\\[1mm]
(3) For $\alpha \in \widehat{X}$, it holds that
\begin{align*}
H^i(X',  {\mu'_b}^* f^* \calf'  \otimes  (f^* L)^{ab} \otimes f^* P_{\alpha}) &= H^i(X', f^* \mu_b^* \calf  \otimes f^* L^{ab} \otimes f^* P_{\alpha}) \\
&=\bigoplus_{ \gamma \in \hat{f}^{-1}(\hat{o}')} H^i(X, \mu_b^* \calf \otimes L^{ab} \otimes P_{\alpha} \otimes P_{\gamma}) ,
\end{align*}
where $ \hat{o}' \in \widehat{X}'$ is the origin and $\hat{f} : \widehat{X} \rightarrow \widehat{X}' $ is the dual isogeny.
Hence $h^i(X, \mu_b^* \calf \otimes L^{ab} \otimes P_{\alpha} ) =0$ for any $\alpha \in  \widehat{X}$
 if and only if $h^i(X',  {\mu'_b}^* f^* \calf'  \otimes  (f^* L)^{ab} \otimes f^* P_{\alpha}) =0$ for any $\alpha \in  \widehat{X}$.
Since $f^* P_{\alpha} =P_{\hat{f} (\alpha)}$ and $\hat{f} $ is surjective,
this is equivalent to the condition that $h^i(X',  {\mu'_b}^* f^* \calf'  \otimes  (f^* L)^{ab} \otimes P_{\alpha'}) =0$ for any $\alpha' \in  \widehat{X}'$.
Hence $\calf \langle xl \rangle$ is IT(0) if and only if  so is $f^* \calf \langle x f^* l \rangle$.\\[1mm]
(4) 
Let $\hat{\iota} : \widehat{X} \rightarrow \widehat{X}'$ be the dual surjection.
For the non-twisted case, $\calf' $ on $X'$ is IT(0) if and only if so is $\iota_* \calf'$  on $X$
since $\hat{\iota} $ is surjective and
\[
h^i(X,  \iota_* \calf' \otimes P_{\alpha}) =h^i(X', \calf' \otimes P_{\alpha}|_{X'}) = h^i(X', \calf' \otimes P_{\hat{\iota}(\alpha)})
\]
for $\alpha \in \widehat{X}$.
Hence we do not need to distinguish $\iota_* \calf'$ and $\calf'$
when we consider the property IT(0) for non-twisted sheaves.
Furthermore,
$\calf'$ is IT(0) if and only if so is $ t_p^* \calf'$  for some $p \in X$,
where $t_p : X \rightarrow X$ is the translation by $p$.

Let $g=\dim X, g'=\dim X'$.
Then $\mu_b^{-1}(X')$ is the disjoint union of $b^{2(g-g')}$ copies $X'_1,\dots, X'_{b^{2(g-g')}}$ of $X'$.
For each $j$,
there exists $p_j \in \mu_b^{-1} (o)$ such that the translation by $p_j$ gives an isomorphism 
$t_j : X' \stackrel{\sim}{\rightarrow} X'_j \subset X$. 
Then
$\mu_b^* \iota_* \calf'  \otimes L^{ab}$ is the direct sum of 
$(\mu_b^* \iota_* \calf'  \otimes L^{ab}) |_{X'_j} $ for $ 1 \leq j \leq b^{2(g-g')}$.
Since $\mu_b |_{X'_j} \circ  t_j  : X' \stackrel{\sim}{\rightarrow} X'_j \rightarrow X' \subset X$ coincides with $\mu'_b$ by $p_j \in \mu_b^{-1} (o)$,
it holds that
\[
t_j^* \left( (\mu_b^* \iota_* \calf'  \otimes L^{ab}) |_{X'_j} \right)= {\mu'_b}^* \calf'  \otimes t_j^* (L^{ab} |_{X'_j}).
\]
Hence $(\mu_b^* \iota_* \calf'  \otimes L^{ab}) |_{X'_j}$ is IT(0) if and only if so is ${\mu'_b}^* \calf'  \otimes  t_j^* (L^{ab} |_{X'_j})$
if and only if so is ${\mu'_b}^* \calf'  \otimes L|_{X'}^{ab} $.
Since $\mu_b^* \iota_* \calf'  \otimes L^{ab}$ is the direct sum of 
$(\mu_b^* \iota_* \calf'  \otimes L^{ab}) |_{X'_j} $,
$\mu_b^* \iota_* \calf'  \otimes L^{ab}$ is IT(0) if and only if so is ${\mu'_b}^* \calf'  \otimes L|_{X'}^{ab} $.
Hence $ \iota_* \calf' \langle x l\rangle$  is IT(0) if and only if so is $\calf' \langle xl|_{X'} \rangle$.
\end{proof}

For any integer $m >0$ and any polarization $l $,
it holds  that $\beta(ml) =m^{-1} \beta(l)$. 
Hence for a rational ample class $\xi \in N^1 (X)_{\Q} :=(\Pic(X) / \Pic^0 (X) ) \otimes_{\Z} {\Q}$,
we can define $ \beta(\xi) := m^{-1} \beta(m\xi)$ for sufficiently divisible $m >0$.
Thus we have a function $\beta : \mathrm{Amp}(X)  \cap N^1 (X)_{\Q} \rightarrow \R$,
where
 $\mathrm{Amp}(X) \subset N^1 (X):=(\Pic(X) / \Pic^0 (X) ) \otimes_{\Z} {\R}$
is the ample cone.

We can show that $\beta $ extends to a continuous function on $ \mathrm{Amp}(X)$,
though we do not use this fact in the rest of the paper:

\begin{prop}\label{lem_L+N}
Let $X $ be an abelian variety.
\begin{enumerate}
\item  Let $l \in \Pic(X)/\Pic^{0}(X)$ be a polarization
and $n \in \Pic(X)/\Pic^{0}(X)$ be a nef (resp.\ ample) class.
Then $\beta(l+n) \leq \beta(l)$ (resp.\ $\beta(l+n) < \beta(l)$).
\item The function 
$\beta : \mathrm{Amp}(X) \cap N^1 (X)_{\Q} \rightarrow \R$
extends to a continuous function on $\mathrm{Amp}(X)$. 
\end{enumerate}
\end{prop}

\begin{proof}
(1) Let $L,N$ be line bundles representing $l,n$ respectively.

For the case when $n$ is nef,
it suffices to show that $x > \beta(l+n) $ for any $x =\frac{a}{b} > \beta(l) $.
Take an integer $m \gg 1$ so that $x=\frac{a}{b} > \frac{ma-1}{mb} > \beta(l) $.
By \autoref{lem_beta_IT(0)},
$\mu_{mb}^* \cali_p \otimes L^{(ma-1)mb}$ is IT(0).
On the other hand, $L^{mb}\otimes N^{m^2ab}  $ is IT(0) since it is ample.
Thus 
\[
\mu_{mb}^* \cali_p \otimes (L \otimes N)^{m^2ab} =(\mu_{mb}^* \cali_p \otimes L^{(ma-1)mb}) \otimes (L^{mb}\otimes N^{m^2ab} )
\]
is IT(0) as well by \cite[Proposition 3.1]{MR2807853}.
Hence $\cali_p  \langle x (l+n) \rangle $ is IT(0) by $x=\frac{ma}{mb}$,
which is equivalent to 
$\beta(l+n) < x$.

If $n$ is ample,
$c n - l$ is nef for some integer $c \gg 1$.
Then it holds that
\[
 \beta(l+n)= c \cdot \beta(c(l+n))= c \cdot  \beta((c+1)l+(cn-l)) \leq c \cdot  \beta((c+1)l)  = \frac{c}{c+1}  \beta(l) < \beta(l),
 \]
 where the first inequality follows from the nef case.
 
 \vspace{1mm}
 \noindent
 (2) It suffices to show that if a sequence $\{\xi_i\}_{i \in \N}$ in $\mathrm{Amp}(X) \cap N^1 (X)_{\Q}$ converges to some $\xi \in \mathrm{Amp}(X)$,
 then $\{\beta(\xi_i)\}_{i \in \N}$ also converges to a real number.

For any rational number $0 < \delta < 1$,
there exists $N \geq 0$ such that
both $(1+\delta) \xi_n - \xi  $ and $ \xi - (1-\delta) \xi_n$ are ample for any $n \geq N$.
Hence $(1+\delta) \xi_n -(1 -\delta) \xi_m$ is ample for any $m,n \geq N$.
By (1),
we have 
\[
(1+\delta)^{-1} \beta(\xi_n) = \beta((1+\delta) \xi_n) < \beta((1-\delta) \xi_m)  = (1-\delta)^{-1} \beta(\xi_m) .
\]
By $m, n \rightarrow \infty$,
\[
(1+\delta)^{-1} \varlimsup_{n \to \infty}  \beta(\xi_n) \leq  (1-\delta)^{-1}  \varliminf_{m \to \infty}   \beta(\xi_m)  .
\]
By $\delta \rightarrow 0$, we have
\[
\varlimsup_{n \to \infty}  \beta(\xi_n) \leq    \varliminf_{m \to \infty}   \beta(\xi_m)  .
\]
Hence $\{ \beta(\xi_n)\}$ converges to a real number.
\end{proof}

\subsection{Singularities and non-klt locus}

We recall some notation in birational geometry. 
Let $X$ be a smooth variety 
and $D=\sum_i d_i D_i$ be an effective $\Q$-divisor on $X$. 
Let $f: Y \arw X $ be a log resolution of $(X,D)$
and write 
\[
K_Y = f^* (K_X +D) + F 
\]
with $F=\sum_j b_j F_j$.
Here we assume that if $F_j$ is the strict transform of $D_i$,
we take $b_j=-d_i$,
and all the other $F_j$'s are exceptional divisors of $f$.

The pair $(X,D)$ is called \emph{log canonical}, or \emph{lc} for short,  (resp.\ \emph{klt}) at $x \in X$ 
if $b_j \geq -1$ (resp.\ $b_j > -1$) for any log resolution $f$ and any $j$ with $x \in f(F_j)$.
The pair $(X,D)$ is called log canonical (resp.\ klt) if $(X,D)$ is log canonical (resp.\ klt) at any $x \in X$.
A prime divisor $F_j$ with $b_j \leq -1$ is called a \emph{non-klt place} of $(X,D)$.
A subvariety $Z \subset X$ is called a \emph{non-klt center} of $(X,D)$
if it is the image of a non-klt place.
When $(X,D) $ is log canonical at the generic point of a non-klt center $Z$, 
$Z$ is also called a \emph{log canonical center}, or \emph{lc center} for short.

The \emph{log canonical threshold} of $(X,D)$ at $x \in X$ is 
\[
\lct_x(X,D)=\lct_x(D) := \max\{ s \geq 0 \, | \, (X,s D) \text{ is log canonical at } x\}.
\]

The \emph{multiplier ideal} $\calj(X,D)$ of $(X,D)$ is defined as
\[
\calj(X,D) := f_* \calo_Y(\lceil F \rceil),
\]
which does not depend on the choice of the log resolution $f$.
Note that $\Supp \calo_X/\calj(X,D)$
coincides with the \emph{non-klt locus} 
\[
\Nklt(X,D) :=\{ x \in X \, | \, (X,D) \text{ is not klt at } x \} .
\]

The following result about the existence of minimal lc centers is known
(see \cite{MR1207013}, \cite{MR1455517}, \cite{MR1457742}).

\begin{thm}\label{thm_minimal _center}
Let $X$ be a smooth variety 
and $D$ be an effective $\Q$-divisor on $X$ such that $(X,D) $ is log canonical.
Then every irreducible component of the intersection of two lc centers of $(X,D)$
is also a lc center of $(X,D)$.

In particular, if $(X,D) $ is log canonical but not klt at $x \in X$,
there exists the unique minimal lc center $Z$ of $(X,D)$ containing $x$.
Furthermore, $Z$ is normal at $x$. 
\end{thm}

The following lemma follows from the standard Tie-breaking trick.

\begin{lem}\label{lem_perturbation}
Let $X$ be a smooth projective variety and $D$ be an effective $\Q$-divisor on $X$ such that
$(X,D)$ is lc but not klt at a point $x \in X$.
Let $Z$ be the minimal lc center of $(X,D)$ at $x$.
Assume that $B$ is an ample $\Q$-divisor on $X$ such that $B-D$ is ample.
Then there exists an effective $\Q$-divisor $F$ on $X$ such that 
$B-F$ is ample and $ \Nklt(X,F) =Z$ in a neighborhood of $x$. 
\end{lem}

\begin{proof}
Let $H \subset X$ be a general effective divisor containing $Z$.
Then we may take $F=(1-\ep) D+ \ep' H$ for $0 < \ep \ll \ep' \ll 1$.
\end{proof}

The following lemma is also well-known.

\begin{lem}\label{lem_restriction_general_fiber}
Let $f : X \rightarrow T$ be a smooth morphism between smooth varieties,
and $D$ be an effective $\Q$-divisor on $X $. 
Set $X_t = f^{-1} (t)$ for $t \in T$.
If $t$ is general, 
it holds that
\[
\Nklt (X_t, D|_{X_t}) = \Nklt (X,D) \cap X_t.
\]
\end{lem}

\begin{proof}
By \cite[Theorem 9.5.35]{MR2095472},
we have $\calj(X_t, D|_{X_t}) = \calj(X,D) |_{X_t} $ if $t$ is general.
This lemma follows by considering the cosupports of the ideals.
\end{proof}

The following lemma is a special case of 
 \cite[Proposition 3.2]{MR1455517},
 and the first step of cutting down lc centers in \cite{MR1455517}, \cite{MR1457742}.

\begin{lem}\label{lem_cut_center_curve}
Let $X$ be a smooth projective variety of dimension $n$, $B$ be an ample $\Q$-divisor on $X$,
and $D \equiv B$ be an effective $\Q$-divisor on $X$.
For a point $x \in X$,
assume that $m:= \mult_x (D) > n$.
Let $c:=\lct_x (D) \leq \frac{n}{m} <1$
and let $Z \subset X$ be the minimal lc center of $(X,cD)$ at $x$.
If 
$(B^{\dim Z}.Z) \geq n^{\dim Z} \cdot \mult_x (Z)$,
there exists an effective $\Q$-divisor $D_1 \equiv c_1 B$ for some $c < c_1 <1$
such that $(X,D_1)$ is lc but not klt at $x$ and the minimal lc center $Z_1$ of $(X,D_1)$ at $x$ is properly contained in $Z$.
\end{lem}

\begin{proof}
Under the assumption of this lemma,
\[
b_x(X,cD) := \sup \{ \mult_x(D') \, | \, D' \text{ is an effective $\Q$-divisor such that }  (X,cD+D') \text{ is lc at } x\}
\]
is at most $n -\mult_x(cD) =n-mc$.
Then the statement of this lemma follows from \cite[Proposition 3.2]{MR1455517}
since the condition 
\[
(B^{\dim Z}.Z)  > \left(\frac{b_x(X,cD)}{1-c}\right)^{\dim Z} \cdot \mult_x(Z)
\]
in \cite[Proposition 3.2]{MR1455517}
 is satisfied by $(B^{\dim Z}.Z) \geq n^{\dim Z} \cdot \mult_x (Z) , n-mc \geq b_x(X,cD)$ and $m >n$.
 We note that $B$ is assumed to be a line bundle in \cite[Proposition 3.2]{MR1455517},
 but the proof works for $\Q$-divisors.
\end{proof}

\begin{rem}\label{rem_cut_center_curve}
In the proof of \autoref{main_thm},
we use \autoref{lem_cut_center_curve}
when $X$ is an abelian threefold and $Z$ is a curve.
In this case,
the condition $(B^{\dim Z}.Z) \geq n^{\dim Z} \cdot \mult_x (Z)$ is equivalent to $ (B.Z) \geq 3$ since $Z$ is smooth at $x$ by \autoref{thm_minimal _center}.
\end{rem}

\section{Bpf thresholds and Seshadri constants}\label{sec_by_seshadri_constant}

For a polarized abelian variety $(X,l)$,
the \emph{Seshadri constant} of $(X,l)$ is defined by
\[
\ep(X,l) =\sup \{ t \geq 0 \, | \, \pi^* l - t E \text{ is nef}\},
\]
where $\pi : \mathrm{Bl}_o X \rightarrow X$ is the blow-up of $X$ at the origin $o$ and $E$ is the exceptional divisor.
Equivalently, we may define
\begin{align}\label{eq_seshadri_def1}
 \ep(X,l) &= \inf_V \sqrt[\dim V]{\frac{(l^{\dim V}.V)}{\mult_o (V)}} ,
\end{align}
where $ \mult_o (V)$ is the multiplicity of $V$ at $o$ and we take the infimum over all subvarieties $V \subset X$ containing $o$ with $\dim V \geq 1$.
We refer the readers to \cite[Section 5]{MR2095471} for detail.

By definition,
$ \ep(X,l)^{-1} < x$ if and only if 
$x \,  \pi^* l -  E $ is ample.
On the other hand, $\beta(l)  < x  $ if and only if $ \cali_o \langle x l \rangle $ is IT(0).
We note that $x \,  \pi^* l -  E $ (resp.\ $ \cali_o \langle x l \rangle $) can be considered as a formal tensor product of $ \calo(-E)  = \pi^{-1} \cali_o$ with $x \,  \pi^* l $ (resp.\ $\cali_o$ with $xl$).
Furthermore, recall that a line bundle $N$ on an abelian variety is IT(0) 
if and only if $N$ is ample.
Thus we might think that the definition of $\beta(l)$ is similar to that of $ \ep(X,l)^{-1}$.
In fact,
the following direct relation between $\beta(l)$ and $ \ep(X,l)^{-1}$ is known:

\begin{prop}[{\cite{MR2833789},\cite{MR4114062},\cite{CaucciThesis}}]\label{prop_bounds_seshadri}
Let $(X,l)$ be a polarized abelian variety of dimension $g$.
Then it holds that $\ep(X,l)^{-1} \leq \beta(l) \leq g \cdot \ep(X,l)^{-1} $.
\end{prop}

\begin{proof}
The first inequality is shown in \cite[Proposition E]{CaucciThesis}.
For the second inequality,
we use a threshold
\begin{equation}\label{eq_r(l)}
\begin{split}
r(l) =\inf \{ t \in \Q \, | \, & \text{ there exists an effective 
$\Q$-divisor $F$ on $X$ such that}\\
& \hspace{40mm} \text{$tl -F$ is ample and } \calj(X,F)=\cali_o  \}
\end{split}
\end{equation}
introduced in \cite{MR4114062}.
By the Nadel vanishing theorem,
Caucci shows $\beta(l) \leq r(l)$ \cite[Proposition 1.4]{MR4114062}.
On the other hand,
\cite[Lemma 1.2]{MR2833789} essentially states that  $ r(l) \leq g  \cdot \ep(X,l)^{-1}$.
Hence we have the second inequality.
\end{proof}

For $\ep(X,l)^{-1}$, the following upper and lower bounds are obtained from the degrees of abelian subvarieties.

\begin{lem}\label{lem_bounds_seshadri}
Let $(X,l)$ be a polarized abelian variety.
Then 
\[
\max_Z \frac{1}{\sqrt[\dim Z]{(l^{\dim Z}.Z)}} \leq \frac{1}{\ep(X,l)} \leq \max_{Z }  \frac{\dim Z}{\sqrt[\dim Z]{(l^{\dim Z}.Z)}},
\]
where we take the maxima over all abelian subvarieties $\{o\} \neq Z \subset X$.
\end{lem}

\begin{proof}
The lower bound follows from \ref{eq_seshadri_def1} since $\mult_o (Z) =1$ for any abelian subvariety $Z$.
The upper bound follows from \cite[Lemma 3.3]{MR1393263} or  \cite[Theorem 1.2]{Ohta_2022}.
We note that $\{ \sqrt[\dim Z]{(l^{\dim Z}.Z)^{-1}} \}_Z$ and $\{ \dim Z \cdot \sqrt[\dim Z]{(l^{\dim Z}.Z)^{-1}} \}_Z$ have maxima since $(l^{\dim Z}.Z)$'s are integers.
\end{proof}

We note that the affirmative answer to \autoref{ques2} is equivalent to 
\[
\beta (l) \leq \max_{Z }  \frac{\dim Z}{\sqrt[\dim Z]{(l^{\dim Z}.Z)}},
\]
where we take the maximum over all abelian subvarieties $\{o\} \neq Z \subset X$.
This upper bound is the same as that in \autoref{lem_bounds_seshadri}.
By \autoref{prop_bounds_seshadri} and \autoref{lem_bounds_seshadri},
we have a partial answer to \autoref{ques2}.

\begin{prop}\label{prop_bounds_any_dim}
Let $(X,l)$ be a polarized abelian variety of dimension $g$.
Then 
\[
\beta(l) \leq \max_{Z }  \frac{g \cdot \dim Z}{\sqrt[\dim Z]{(l^{\dim Z}.Z)}},
\]
where we take the maximum over all abelian subvarieties $\{o\} \neq Z \subset X$.

Equivalently,
$\beta(l) < t$ holds for $ t >0$ if 
$(B_t^{\dim Z} . Z) > (g \cdot \dim Z)^{\dim Z}$ for any abelian subvariety $\{o\} \neq Z \subset X$,
where $B_t :=tl$.
\end{prop}

\begin{proof}
This is a direct consequence of \autoref{prop_bounds_seshadri} and \autoref{lem_bounds_seshadri}.
\end{proof}

For lower bounds of $\beta(l)$,
Caucci \cite[Proposition 1.6.10]{CaucciThesis} obtains the inequality $ \beta(l) \geq \sqrt[g]{(l^g)^{-1}}$ from $\beta(l) \geq \ep(X,l)^{-1}$.
Similarly,
we also have $\beta(l) \geq \max_Z  \sqrt[\dim Z]{(l^{\dim Z}.Z)^{-1}} $
by $\beta(l) \geq \ep(X,l)^{-1}$ and \autoref{lem_bounds_seshadri}.
The following lemma improves these lower bounds
by the factor $\sqrt[\dim Z]{(\dim Z)!} $ for each $Z$. 
Since $\beta(l) < x$ if and only if $\beta(xl) <1$,
this lemma could be considered as the statement that if a $\Q$-line bundle $xL$ is ``basepoint free'', ``$h^0(xL) =$''$ x^g \chi(l)$ is greater than one,
and
$x L|_Z$ is ``basepoint free'' for any abelian subvariety $Z \subset X$.

\begin{lem}\label{lem_lower_bound}
Let $(X,l)$ be a polarized abelian variety of dimension $g$.
Then it holds that
\begin{enumerate}
\item $\beta(l) \geq \sqrt[g]{\chi( l )^{-1}}  = \sqrt[g]{g!} \sqrt[g]{(l^g)^{-1}}  $,
\item $\beta(l) \geq \beta(l|_Z)$ for any abelian subvariety $\{o\} \neq Z \subset X$.
\end{enumerate}
As a consequence, it holds that
\begin{align}\label{eq_lower_bound_Y}
\beta(l) \geq  \max_Z \frac{1}{\sqrt[\dim Z]{\chi( l |_Z)}} = \max_Z \sqrt[\dim Z]{\frac{(\dim Z)!}{(l^{\dim Z}.Z)}} ,
\end{align}
where we take the 
maxima over all abelian subvarieties $\{o\} \neq Z \subset X$.
\end{lem}

\begin{proof}
\ref{eq_lower_bound_Y} follows from (1), (2)
since we have
$ \beta(l|_Z) \geq \sqrt[\dim Z]{\chi( l |_Z)^{-1}} $ by applying (1)  to $(Z,l|_Z)$.
Hence it suffices to show (1) and (2).

\vspace{1mm}
\noindent
(1) Take a rational number $x=\frac{a}{b}  > \beta(l)$. 
It is enough to show $x \geq  \sqrt[g]{\chi(l)^{-1}}   $.

By the choice of $x$,
$ \mu_b^* \cali_o \otimes L^{ab}$ is IT(0)
and hence
\[
H^0(X, L^{ab}  ) \rightarrow H^0( X, L^{ab}  \otimes \calo_X/\mu_b^* \cali_o)
\]
is surjective. 
Thus we have
\begin{align*}
 (ab)^g \chi(l)=\chi(L^{ab} ) &=h^0(X, L^{ab}   ) \\
 &\geq h^0( X, L^{ab}  \otimes \calo_X/\mu_b^* \cali_o) = \deg \mu_b = b^{2g}
\end{align*}
and hence $ x^g = (a/b)^g \geq \chi (l)^{-1}$.

\vspace{1mm}
\noindent
(2)
Take a rational number $x=\frac{a}{b}  > \beta(l)$. 
It suffices to show that $\cali_{o/Z} \langle x l|_Z \rangle$ is IT(0),
where $\cali_{o/Z} $ is the ideal sheaf on $Z$ corresponding to  $o \in Z \subset X$.

Consider the diagram
  \[
  \xymatrix{
H^0(X, L^{ab} \otimes P_{\alpha} ) \ar[d]_{R} \ar[r]^(.4){r} \ar@{}[dr] &  H^0( X, L^{ab} \otimes P_{\alpha} \otimes \calo_X/\mu_b^* \cali_o) \ar[d]_{R'} \\
     H^0(Z, L|_Z^{ab} \otimes P_{\alpha}|_Z )   \ar[r]^(.38){r'} & H^0(Z, L|_Z^{ab} \otimes P_{\alpha}|_Z \otimes  \calo_Z/\mu_{b,Z}^* \cali_{o/Z} )  \\
  }
  \]
where $\alpha \in \widehat{X}$ and $\mu_{b,Z} : Z \rightarrow Z $ is the multiplication-by-$b$ isogeny on $Z$.
Since $ \mu_b^* \cali_o \otimes L^{ab}$ is IT(0),
$r$ is surjective for any $\alpha$.
Since $R'$ is induced from the surjection $\calo_X/\mu_b^* \cali_o \rightarrow \calo_Z/\mu_{b,Z}^* \cali_{o/Z}$ between skyscraper sheaves,
$R'$ is surjective as well.
Hence $r'$ is also surjective and hence we have $h^1(Z,  \mu_{b,Z}^* \cali_{o/Z}   \otimes L|_Z^{ab} \otimes P_{\alpha}|_Z )=0$.
Since the natural homomorphism $ \widehat{X} \rightarrow \widehat{Z}$ is surjective, 
$ \mu_{b,Z}^* \cali_{o/Z}   \otimes L|_Z^{ab}$ is IT(0).
Hence so is $\cali_{o/Z} \langle x l|_Z \rangle$. 
\end{proof}

If \autoref{ques2} has an affirmative answer,
we have
\begin{align}\label{eq_upper_and_lower}
\max_Z \frac{\sqrt[\dim Z]{(\dim Z)!}}{\sqrt[\dim Z]{(l^{\dim Z}.Z)}} \leq \beta(l) \leq \max_{Z }  \frac{\dim Z}{\sqrt[\dim Z]{(l^{\dim Z}.Z)}}
\end{align}
by \autoref{lem_lower_bound}.
These inequalities are similar to those for $\ep(X,l)^{-1}$ in \autoref{lem_bounds_seshadri}.

\vspace{2mm}
By the above similarity, 
we might expect $ \beta(l)^{-1}$ and $\ep(X,l)$ could have similar properties.
For example, both $ \beta(l)^{-1}$ and $\ep(X,l)$ are homogeneous,
that is,
$\beta(m l)^{-1} =m \beta(l)^{-1}$ and $\ep(X,m l) =m \ep(X,l) $  for any positive integer $m$.
It is known that the Seshadri constant of a direct product $(X \times X' , l \boxtimes l' )$
is the minimum of $\ep(X , l )$ and $ \ep( X' , l')$ (see \cite[Proposition 3.4]{MR3338009} for example).
Bpf thresholds also have this propoerty:

\begin{lem}\label{lem_ep_of_product}
Let $(X,l), (X',l')$ be polarized abelian varieties.
Then $\beta(X \times X' , l \boxtimes l') =\max \{\beta(X , l ), \beta( X' , l') \}$.
\end{lem}

\begin{proof}
Let $o \in X, o' \in X'  $ 
be the origins.
Since $(X,l) ,(X',l')$ are polarized abelian subvarieties of $(X \times X' , l \boxtimes l') $,
we have $\beta(X \times X' , l \boxtimes l')  \geq  \max \{\beta(X , l ), \beta( X' , l') \}$ by \autoref{lem_lower_bound}.

Take a rational number $x > \max \{\beta(X , l ), \beta( X' , l') \}$.
To prove the converse inequality $\beta(X \times X' , l \boxtimes l')  \leq  \max \{\beta(X , l ), \beta( X' , l') \}$,
it suffices to show $\beta(X \times X' , l \boxtimes l')  <  x $, that is,
$\cali_{(o,o')} \langle x(l \boxtimes l') \rangle $ is IT(0).

Consider the exact sequence
\[
0 \rightarrow \cali_{\{o\} \times X'} \rightarrow \cali_{(o,o')} \rightarrow \cali_{(o,o')/\{o\} \times X'} \rightarrow 0
\]
on $X \times X'$.
Both $\cali_{o}  \langle x l  \rangle $ and $ \calo_{X'}  \langle x l'  \rangle$ are IT(0) by $ \beta(X , l ) < x$.
Hence $\cali_{\{o\} \times X'}   \langle x(l \boxtimes l') \rangle $ is IT(0) by \autoref{lem_IT(0)} (2).
By $\beta( X' , l')  < x$, the $\Q$-twisted sheaf $\cali_{o'} \langle x l'  \rangle$ on $X'$ is IT(0).
Thus the $\Q$-twisted sheaf $\cali_{(o,o')/\{o\} \times X'}  \langle x(l \boxtimes l') \rangle $
is IT(0) by \autoref{lem_IT(0)} (4).
Hence $\cali_{(o,o')} \langle x(l \boxtimes l') \rangle$ is IT(0) by \autoref{lem_IT(0)} (1).
 \end{proof}

Another basic property of Seshadri constant is the inequality
\[
\ep(X,l+n) \geq \ep(X,l) + \ep(X,n)
\]
for two polarizations $l,n$ on $X$.
Hence it would be interesting to ask if 
\[
\beta( l+n )^{-1} \geq \beta( l )^{-1}  +\beta( n)^{-1} 
\]
holds or not in general.

\section{Bpf thresholds on abelian surfaces}\label{sec_surface}

In this section,
we study Bpf thresholds on abelian surfaces.
First,
we show that the inequalities \ref{eq_upper_and_lower} hold for abelian surfaces.
We note that the upper bounds of $\beta(l)$ in the following proposition are stated in \cite[p.950]{MR4114062}.

\begin{prop}\label{prop_bounds_surface}
Let $(X,l)$ be a polarized abelian surface
and set
\[
\mu(X,l) =\sup \{ t \geq 0 \, | \, \pi^* l - t E \text{ is big}\},
\]
where $\pi : \mathrm{Bl}_o X \rightarrow X$ is the blow-up of $X$ at the origin and $E$ is the exceptional divisor.
Let $r(l) \in \R$ be the threshold defined in  \ref{eq_r(l)}.
Then it holds that
\begin{align}\label{eq_bounds_surface}
\max \left\{ \frac{\sqrt{2}}{\sqrt{(l^2)}}, \max_{C} \frac{1}{(l.C)} \right\} \leq \beta(l)  \leq r(l) \leq \max \left\{ \frac{2}{\mu(X,l)}, \max_{C} \frac{1}{(l.C)} \right\}  \leq  \max \left\{ \frac{2}{\sqrt{(l^2)}}, \max_{C} \frac{1}{(l.C)} \right\},
\end{align}
where we take the maxima over all elliptic curves $C \subset X$.
In particular,
\ref{eq_upper_and_lower} holds for abelian surfaces.
\end{prop}

\begin{proof}
The first inequality follows from \autoref{lem_lower_bound}.
The second inequality is shown by \cite[Proposition 1.4]{MR4114062}.
The third inequality follows from \cite[Proposition 3.1]{MR3923760}.
Since $ \mu(X,l) \geq \sqrt{(l^2)}$ by the Riemann-Roch theorem, the last inequality follows.
\end{proof}

\begin{rem}\label{rem_bound_mu}
If $\sqrt{(l^2)}$ is not an integer,
we have a lower bound of $ \mu(X,l) $ which is slightly better than $\sqrt{(l^2)}$ as follows.

We may assume that $l$ is primitive, i.e.\ 
$l$ is not written as $kn$ for some integer $k \geq 2$ and some polarization $n$.
Let $(1,d)$ be the type of $l$.
In particular, $\sqrt{2d}=\sqrt{(l^2)} $ is not an integer.
Let $(k_0,m_0) $ be the primitive solution of Pell's equation $m^2-2dk^2=1$.
By the proof of \cite[Theorem A.1, (a)]{MR1660259},
there exists an effective divisor $D \equiv 2k_0 l$ such that $\mult_o (D)= 2m_0$.
Hence we have 
\[
\mu(X,l) \geq   \frac{m_0}{k_0} =\sqrt{2d + \frac{1}{k_0^2}} >\sqrt{2d}= \sqrt{(l^2)}.
\]
If the Picard number of $X$ is one, we have $\mu(X,l) =  m_0/k_0$
by \cite[Theorem 6.1, (b)]{MR1678549}.
\end{rem}

If $ 2/\sqrt{(l^2)} \leq  \max_{C}  1/(l.C) $,
both the upper and lower bounds of \ref{eq_bounds_surface} are $ \max_{C} 1/(l.C) $ 
and hence we have $\beta(l) =  \max_{C} 1/(l.C) $.
We show that this equality 
holds
under a weaker condition $ \sqrt{2}/\sqrt{(l^2)} \leq  \max_{C} 1/(l.C) $.
We use the following lemma:

\begin{lem}\label{lem_small_deg_hypersurface}
Let $(X,l)$ be a polarized abelian $g$-fold
and $S \subset X$ be an abelian subvariety of codimension one.
Then it holds that
\[
\beta(l|_S) \leq \beta(l) \leq \max \left\{ \beta(l|_S), \frac{g (l^{g-1}.S)}{(l^g)} \right\} =\max \left\{ \beta(l|_S), \frac{\chi(l|_S)}{\chi(l)} \right\} .
\]
\end{lem}

\begin{proof}
The first inequality follows from \autoref{lem_lower_bound} (2).

Take a rational number $x > \max\{ \beta(l|_S), g (l^{g-1}.S)/(l^g) \} $.
To prove the second inequality,
it suffices to show that $ \cali_o  \langle xl \rangle $ is IT(0).

By $x > g (l^{g-1}.S)/(l^g)$,
we have $((xl)^g) > g \cdot ((xl)^{g-1}.S) $.
Hence $xl -  S$ is big by \cite[Theorem 2.2.15]{MR2095472}.
Since $X$ is an abelian variety,
$xl - S$ is ample.

Consider an exact sequence
\[
0 \rightarrow \cali_S \rightarrow  \cali_o \rightarrow \cali_{o/S} \rightarrow 0
\]
on $X$.
Since $xl-S$ is ample,
$\cali_S  \langle xl \rangle = \calo_X(-S)  \langle xl \rangle$ is IT(0).
By $x >  \beta(l|_S)$,
$\cali_{o/S}  \langle xl|_S \rangle$ is IT(0) as a $\Q$-twisted sheaf on $S$.
Hence $\cali_{o/S}  \langle xl \rangle$ is IT(0) as a $\Q$-twisted sheaf on $X$ by \autoref{lem_IT(0)} (4).
Then $ \cali_o  \langle xl \rangle$ is IT(0)
by \autoref{lem_IT(0)} (1).
\end{proof}

\begin{prop}\label{prop_surface}
Let $(X,l)$ be a polarized abelian surface.
Assume that there exists 
an elliptic curve $C \subset X$ such that $ \sqrt{(l^2)}/ \sqrt{2} \geq  (l.C)$.
Then it holds that
\[
\beta(l)= \ep(X,l)^{-1} =\frac{1}{(l.C)} .
\]
\end{prop}

\begin{proof}
By \autoref{rem_curve}, we have $\beta(l|_C)  =\deg(l|_C)^{-1}=(l.C)^{-1}$.
By \autoref{lem_small_deg_hypersurface},
\[
\frac{1}{(l.C)} =\beta(l|_C) \leq \beta(l) \leq \max \left\{ \beta(l|_C), \frac{2 (l.C)}{(l^2)} \right\} = \max \left\{ \frac{1}{(l.C)}, \frac{2 (l.C)}{(l^2)} \right\} =\frac{1}{(l.C)},
\]
where the last equality follows from
the assumption $ \sqrt{(l^2)}/ \sqrt{2} \geq  (l.C)$.
Hence we have $\beta(l) = (l.C)^{-1}$.

For the Seshadri constant, it holds that
\[
(l.C) = \frac{1}{\beta(l)} \leq \ep(X,l) \leq (l.C),
\]
where the first inequality follows from \autoref{prop_bounds_seshadri} and the second inequality follows from \ref{eq_seshadri_def1}.
Thus we have $\ep(X,l) = \beta(l)^{-1} =(l.C)$.
\end{proof}

\autoref{prop_surface} states that we can compute $\beta(l)$ when $X$ contains an elliptic curve whose degree $(l.C)$ is small
compared to the volume $(l^2)$ as in the following example.
As far as the author knows,
this is the first example to compute $\beta(l)$ other than (the multiples of) non-basepoint free polarizations.

\begin{ex}\label{ex_beta}
Let $X$ be an abelian surface which contains an elliptic curve $C$.
Let $l'$ be a polarization on $X$
and set $l=l'+a [C]$ for $a \geq (l'.C)$,
where $[C] \in \Pic X / \Pic^0 X$ is the class of $C$.
Then $ \sqrt{(l^2)} \geq \sqrt{2} (l.C)$ and 
hence we have
 $\beta(l)=(l.C)^{-1}=(l'.C)^{-1}$
by \autoref{prop_surface}.
\end{ex}

\begin{rem}
We note that the inequalities in \ref{eq_bounds_surface} are sharp in the sense that
$\beta(l) $ could coincide with $\sqrt{2}/\sqrt{(l^2)}$, $2/\sqrt{(l^2)}$ and $\max_{C} (l.C)^{-1} $ respectively
as follows:

If $l$ is a principal polarization,
$(l^2)=2$ and any $L$ representing $l$ is not basepoint free.
Hence
$\beta(l) =1 =\sqrt{2}/\sqrt{(l^2)}$. 
If $(l^2) =4$, $L$ is not basepoint free since $h^0(L)=2$.
Hence $\beta(l) =1 =2/\sqrt{(l^2)}$.
An example which satisfies $\beta(l) =\max_{C} (l.C)^{-1} $ is given in \autoref{ex_beta}.
\end{rem}

For the next proposition,
we recall three notions on higher order embeddings.
A line bundle $N$ on a smooth projective variety $M $ is called 
$k$-\emph{jet ample} for an integer $k \geq 0$
if the restriction map 
\[
H^0 ( L ) \rightarrow H^0 ( L \otimes \calo_M/ \cali_{p_1}\cali_{p_2} \cdots \cali_{p_{k+1}}) 
\]
is surjective for any (not necessarily distinct) $k+1$ points $p_1,\dots,p_{k+1} \in M$.
A line bundle $N$ is called 
$k$-\emph{very ample} (resp.\ $k$-\emph{spanned})
if the restriction map $H^0(M,N) \rightarrow H^0(Z,N|_Z)$ is surjective for any $0$-dimensional subscheme $Z $ 
(resp.\ for any curvilinear $0$-dimensional subscheme $Z$)
of $M$
with $\mathrm{length} (\calo_Z )= k+1$. 
It is known that $k$-jet ampleness implies $k$-very ampleness
(cf.\ \cite[Proposition 2.2]{MR1211891}),
which of course implies $k$-spannedness.

On abelian varieties,
these notions
are studied in \cite{MR1439201}, \cite{MR1376538}, \cite{MR1654705}, \cite{MR2008719}, etc.
In \cite[Theorem D]{CaucciThesis}, Caucci shows that 
an ample line bundle $L$ on an abelian variety is $k$-jet ample if $\beta(l) < 1/(k+1)$.
On abelian surfaces,
these notions on higher order embeddings are completely determined by $\beta(l)$ or $\ep(X,l)$ under a suitable assumption:

\begin{prop}\label{prop_equivalence}
Let $t >0 $ be a rational number
and $L$ be an ample line bundle on an abelian surface $X$ such that $(L^2) > 4/t^2$.
Then the following are equivalent:
\begin{itemize}
\item[(a)] $\beta(l) < t$,
\item[(b)] $\ep(X,l) > t^{-1}$,
\item[(c)] $(l.C) > t^{-1}$ for any elliptic curve $C \subset X$,
\end{itemize}
where $l=[L]$ is the polarization represented by $L$.

\vspace{1mm}
If $t=1/(p+2) $ for an integer $p \geq -1$,
(in particular, we assume $(L^2) > 4(p+2)^2$),
(a)-(c) and the following are equivalent: 
\begin{itemize}
\item[(d)] $L$ is $(p+1)$-spanned,
\item[(e)] $L$ is $(p+1)$-very ample,
\item[(f)] $L$ is $(p+1)$-jet ample.
\end{itemize}

\vspace{1mm}
Moreover,
if $p \geq 0$,
(a)-(f) are equivalent to 
\begin{itemize}
\item[(g)] $L$ satisfies $(N_p)$.
\end{itemize}
\end{prop}

\begin{proof}
(a) $\Rightarrow$ (b) follows from $\beta(l) \geq \ep(X,l)^{-1}$.
(b) $\Rightarrow$ (c) follows from \ref{eq_seshadri_def1}.
(c) $\Rightarrow$ (a) follows from \ref{eq_bounds_surface}
and the assumption $(l^2) > 4/t^2$.

Assume $t=1/(p+2) $ for an integer $p \geq -1$.
(a) $\Rightarrow$ (f) follows from \cite[Theorem D]{CaucciThesis}.
As we already see,
(f) $\Rightarrow$ (e) and (e) $\Rightarrow$ (d) hold in general.

If there exists an elliptic curve $C \subset  X$ such that $(L.C) \leq t^{-1}=p+2$,
it is easy to find a subscheme $Z \subset C$ with $\mathrm{length} (\calo_Z )= p+2$
such that $ H^0(C, L|_C) \rightarrow H^0(Z,L|_Z)$ is not surjective
(see \cite[Proposition 1.2]{MR1202811} for example).
Then $ H^0(X, L) \rightarrow H^0(Z,L|_Z)$ is not surjective as well
and hence $L$ is not $(p+1)$-spanned.
Thus we have (d) $\Rightarrow$ (c)
and the equivalence of (a)-(f) follows.

If $p \geq 0$,
(c) $\Leftrightarrow$ (g) is nothing but the main result of \cite{MR4009173}, \cite{MR3923760}.
\end{proof}

We note that the most parts of the equivalence of (b)-(g) are already known. 
For example, (b) $\Leftrightarrow$ (c) follows from \cite[Theorem A.1 (b)]{MR1660259}. 
(c) $\Leftrightarrow$ (d)  $\Leftrightarrow$ (e) follows from \cite[Theorem 1.1]{MR1654705}.
As stated in the proof,
(c) $\Leftrightarrow$ (g) is proved in \cite{MR4009173}, \cite{MR3923760}.

\begin{rem}
We see that the assumption $(L^2) > 4/t^2$ is sharp at least for $t=1/(p+2)$ with $p=-1$ or $0$
as follows.

Let $N$ be an ample symmetric line bundle on a simple abelian surface $X$
such that $(N^2) =4$.
Since $h^0(N)=2$,
$N$ is not basepoint free 
and hence $\beta(n)=1$.

For $p=-1$ and $L=N$, it holds that $(L^2)=4 = 4/t^2$ since $t=1/(p+2)=1$.
In this case, (a), (d)-(f) do not hold. 
On the other hand,
(c) holds since $X$ is simple.

For $p=0$ and $L=2N$, it holds that $(L^2)=16 = 4/t^2$ since $t=1/(p+2)=1/2$.
In this case, $L$ is not projectively normal by \cite{MR966402}, 
\cite[Remark 3.6, 2)]{MR1974682},
and $\beta(l)=\beta(n)/2 =1/2$. 
Thus (a), (g) do not hold, 
but (c) holds since $X$ is simple.
\end{rem}

\section{An observation}
\label{sec_by_multiplier_ideal}

Higher syzygies on abelian surfaces and threefolds are studied in \cite{MR4009173}, \cite{MR3923760}, \cite{lozovanu2018singular}.
In these papers, the starting point is the following theorem in \cite{MR2833789}: 

\begin{thm}[\cite{MR2833789}]\label{thm_LPP}
Let $p$ be a non-negative integer,
$X$ be an abelian variety,
and $L$ be an ample line bundle on $X$.
Assume that there exists an effective 
$\Q$-divisor $F$ on $X$ such that
\begin{itemize}
\item[(i)] $\frac{1}{p+2}L-F$ is ample,
\item[(ii)] the multiplier ideal $\calj(X,F)$ coincides with the maximal ideal 
$\cali_o$ of the origin $o \in X$.
\end{itemize}
Then $L$ satisfies property $(N_p)$.
\end{thm}

Note that we may replace (ii) by $\Nklt(X,F) =\{o\}$ as in \cite[Remark 2.1]{MR3923760}.
Hence \autoref{thm_LPP} closely resembles the following well-known fact which is used in the study of Fujita's basepoint freeness conjecture
 in \cite{MR1207013}, \cite{MR1233485}, \cite{MR1455517}, \cite{MR1457742}, \cite{MR4108218},  etc.

\begin{prop}\label{prop_adjoint_ Bpf}
Let $L$ be an ample line bundle on a smooth projective variety $X$
and $x \in X$ be a point.
If there exists an effective $\Q$-divisor $F$ on $X$
such that $ L-F$ is ample and $x$ is an isolated point of $\Nklt(X,F)$,
then $K_X+L$ is basepoint free at $x$.
\end{prop}

A difference between \autoref{thm_LPP} and \autoref{prop_adjoint_ Bpf} is 
that the former assumes the equality $\Nklt(X,F)=\{o\}$ on \emph{whole} $X$
but the latter assumes $\Nklt(X,F)=\{x\}$ \emph{locally}.
As stated in \cite[Section 7]{lozovanu2018singular},
this difference makes the construction of $F$ in \autoref{thm_LPP} harder even in the case $\dim X=3$.
To construct a divisor $F$ in \autoref{thm_LPP} on abelian threefolds,
Lozovanu \cite{lozovanu2018singular} gives a detailed analysis on the restricted volume $\mathrm{vol}_{X'|E} ( \pi^* L -t E)$
using (infinitesimal) Newton-Okounkov bodies and Seshadri constants,
where $\pi : X' \rightarrow X$ is the blow-up at the origin and $E$ is the exceptional divisor.

\vspace{2mm}

An easy but important observation is the following proposition,
whose proof is essentially the same as that of \autoref{prop_adjoint_ Bpf} 
(see \cite[10.4.C]{MR2095472} for example).

\begin{prop}\label{prop_local_enough'}
Let $(X,l)$ be a polarized abelian variety and $t >0$ be a positive rational number.
Let $S \subset X$ be a finite set
and assume that there exists
an effective  $\Q$-divisor $F$ on $X$ such that
\begin{itemize}
\item[(i)] $t l-F$ is ample,
\item[(ii)] each point in $S$
is an isolated point of the non-klt locus $\Nklt(X,F)$.
\end{itemize}
If an ideal sheaf $\cali$ on $X$ satisfies $\calj(X,F) \subset \cali$ and $ \Supp \calo_X / \cali \subset S $,
then $\cali \langle tl \rangle$ is IT(0).
In particular,
$\cali_S \langle tl \rangle$ is IT(0),
where $\cali_S$ is the ideal sheaf of $S$ with the reduced scheme structure.
\end{prop}

\begin{proof}
Let $t=\frac{a}{b}$.
Since $ \mu_b^* \calj(X,F)  = \calj(X, \mu_b^* F)$ and $ab L -  \mu_b^* F \equiv  ab L -b^2 F \equiv b^2 (t l -F)$ is ample,
$h^1(L^{ab } \otimes  \mu_b^* \calj(X,F) \otimes P_{\alpha}) = 0$ for any $\alpha \in \widehat{X}$ by the Nadel vanishing theorem.
Hence the natural map
\[
H^0 (L^{ab } \otimes P_{\alpha}) \rightarrow H^0 (L^{ab } \otimes P_{\alpha} \otimes \calo_X/  \mu_b^* \calj(X,F))
\]
is surjective.

By (ii), there exists an ideal sheaves $\calj, \calj' $ on $X$ such that 
$\calj(X,F) = \calj \cap \calj'$, $\Supp \calo_X/ \calj =S$ and $S \cap \Supp \calo_X / \calj' =\emptyset$.
Then we have
\[
\calo_X/  \mu_b^* \calj(X,F ) =\calo_X/ \mu_b^* (\calj  \cap \calj' ) = \calo_X/ \mu_b^* \calj \oplus \calo_X/ \mu_b^* \calj' 
\]
since $\Supp \calo_X/ \mu_b^* \calj \cap \Supp \calo_X/ \mu_b^* \calj' = \emptyset$.
Thus 
\[
H^0 (L^{ab } \otimes P_{\alpha} \otimes \calo_X/  \mu_b^* \calj(X,F)) 
=H^0 (L^{ab } \otimes P_{\alpha} \otimes \calo_X/  \mu_b^* \calj)  \oplus H^0 (L^{ab } \otimes P_{\alpha} \otimes \calo_X/ \mu_b^* \calj' )
\]
and hence the natural map
\[
 H^0 (L^{ab } \otimes P_{\alpha}) \rightarrow H^0 (L^{ab } \otimes P_{\alpha} \otimes \calo_X/  \mu_b^* \calj ) 
 \]
 is surjective as well.

 The assumption on $\cali$ is equivalent to $\calj \subset \cali$.
 Since $\calo_X/  \mu_b^* \calj$ is a skyscraper sheaf,
 the natural map
 \[
 H^0 (L^{ab } \otimes P_{\alpha} \otimes \calo_X/  \mu_b^* \calj )  \rightarrow 
 H^0 ( L^{ab } \otimes P_{\alpha} \otimes \calo_X/  \mu_b^* \cali ) 
 \]
  is also surjective.
Thus $ H^0 (L^{ab } \otimes P_{\alpha}) \rightarrow  H^0 ( L^{ab}  \otimes P_{\alpha} \otimes \calo_X/  \mu_{b}^* \cali) $ is surjective
and hence $H^1( \mu_{b}^* \cali \otimes L^{ab } \otimes P_{\alpha}  ) =0$ by $H^1(L^{ab } \otimes P_{\alpha})=0$.
Since $\alpha $ can be any point in $\widehat{X}$, $ \mu_{b}^* \cali \otimes L^{ab }$ is IT(0).
Thus $\cali \langle tl \rangle$ is IT(0).

The last statement follows from $\calj(X,F) \subset \cali_S $ and $\Supp \calo_X /\cali_S =S$.
\end{proof}

As a special case,  we obtain the following colollary

\begin{cor}\label{prop_local_enough}
Let $(X,l)$ be a polarized abelian variety and $t >0$ be a positive rational number.
Assume that there exists an effective 
$\Q$-divisor $F$ on $X$ such that
\begin{itemize}
\item[(i)] $tl-F$ is ample,
\item[(ii)] $o \in X$ is an isolated point of $\Nklt(X,F)$.
\end{itemize}
Then $\beta(l) < t$.
In particular, any $L$ representing $l$  satisfies $(N_p)$ if $t \leq \frac{1}{p+2}$.
\end{cor}

\begin{proof}
Recall that $\beta(l) < t$ if and only if  $\cali_o \langle tl \rangle$ is IT(0).
Hence this corollary is a special case $S=\{o\}$ of \autoref{prop_local_enough'}.
\end{proof}

We might regard the condition $\beta(l) < t$, which is equivalent to $ \beta(tl) <1$,
as ``the basepoint freeness of the $\Q $-line bundle $t L$'' by \autoref{thm_ Bpf_threshold}.
Since $tL=K_X+tL$ for an abelian variety, 
\autoref{prop_local_enough} can be regarded as a $\Q$-line bundle case of \autoref{prop_adjoint_ Bpf} on abelian varieties.

\vspace{2mm}

Some techniques to construct $F$ as in \autoref{prop_adjoint_ Bpf} 
are developed in the study of Fujita's conjecture.
The basic strategy is to cut down minimal lc centers
so that the centers become zero-dimensional. 
The strategy and techniques to cut down lc centers are used in \cite{MR3923760}, \cite{lozovanu2018singular}
to construct $F$ as in \autoref{thm_LPP}.
However we can use them more directly 
due to \autoref{prop_local_enough}
since we only need to consider the singularities of $F$ locally as in \autoref{prop_adjoint_ Bpf}.
Recall a threshold $r(l)$ in \ref{eq_r(l)}:
\begin{align*}
r(l) =\inf \{ t \in \Q \, | \, & \text{ there exists an effective 
$\Q$-divisor $F$ on $X$ such that}\\
& \hspace{40mm} \text{$tl -F$ is ample and } \calj(X,F)=\cali_o  \}.
\end{align*}
We define a modified threshold:

\begin{defn}\label{def_condition}
Let $(X,l)$ be a polarized abelian variety.
We define 
\begin{align*}
r'(l) =\inf \{ t \in \Q \, | \, & \text{ there exists an effective 
$\Q$-divisor $F$ on $X$ such that} \\
& \hspace{15mm} \text{$tl-F$ is ample and $\Nklt(X,F)$ has an isolated point}  \}.
\end{align*}
\end{defn}

\begin{rem}\label{rem_not_minimum}
If there exists an effective 
$\Q$-divisor $F$ on $X$ such that $tl-F$ is ample and $\Nklt(X,F)$ has an isolated point,
then $r'(l) \leq t$ by definition.
In fact, we have $r'(l) < t$ since $t'l -F$ is still ample if $0 < t-t' \ll 1$.
\end{rem}

The following corollary is a restatement of  \autoref{prop_local_enough}.

\begin{cor}\label{cor_epsilon_rho}
Let $(X,l)$ be a polarized abelian variety.
Then $\beta(l) \leq r'(l) \leq r(l)$.
\end{cor}

\begin{proof}
We obtain the first inequality  by applying \autoref{prop_local_enough'} to $S=\{p\}$,
where $p$ is an isolated point of $\Nklt(X,F)$.
The second one follows from the definitions of $ r'(l)$ and $r(l)$.
\end{proof}

By Corollaries \ref{prop_local_enough}, \ref{cor_epsilon_rho} and results in the study of Fujita's conjecture, 
we obtain the following proposition,
though we do not use this in the proof of \autoref{main_thm}:

\begin{prop}\label{lem_rho_dim2}
Let $(X,l)$ be a polarized abelian variety and set $B_t=tl$ for $t >0$.
\begin{enumerate}
\item The case $\dim X=1$:  $\beta(l) = r'(l) = r(l) =(\deg l )^{-1}$.
\item The case $\dim X=2$ (\cite{MR4114062}): 
If $(B_{t}^2) >4$ and $(B_{t}.C) >1$ for any elliptic curve  $C \subset X$,
then $\beta(l) \leq r'(l) \leq r(l) < t$.
\item The case $\dim X=3$:
Assume 
 \begin{itemize}
\item $(B_{t}^3) > 27$,
\item $(B_{t}^2.S) \geq 9$ for any surface $S \subset X$, 
\item $(B_{t}.C) \geq 3$ for any curve $C \subset X$. 
\end{itemize}
Then $\beta(l) \leq r'(l) < t$.
\item The case $\dim X=4$:
Assume 
\[
(B_{t}^{\dim Z}.Z) \geq 5^{\dim Z}
\]
for any subvariety $Z \subset X$.
Then $\beta(l) \leq r'(l) < t$.
\item In any dimension:
Assume
\[
(B_{t}^{\dim Z}.Z) > \binom{\dim X+1}{2}^{\dim Z}
\]
for any subvariety $Z \subset X$.
Then $\beta(l) \leq r'(l) < t$.
\end{enumerate}
\end{prop}

\begin{proof}
(1) We have $\beta(l) = (\deg l)^{-1}$ by \autoref{rem_curve}.
On the other hand, $r(l) = (\deg l)^{-1}$ can be checked from the definition.
Hence (1) follows from \autoref{cor_epsilon_rho}.\\
(2) The inequality $r(t) < t$ follows from \autoref{prop_bounds_surface}.
In fact,
this is proved in \cite{MR4114062} and we use his argument in the proof of \autoref{prop_bounds_surface}.\\
(3) \cite[Example 4.6]{MR1686947} shows that the numerical conditions in (3) imply the existence of 
an effective $\Q$-divisor $D \equiv cB_{t}$ on $X$ for some $0 < c <1$
such that $o$ is the minimal lc center of $(X,D)$.
Perturbing $D$ slightly by \autoref{lem_perturbation},
we obtain $F$ as in the definition of $r'(l)$. 
Hence we have $r'(l) < t$.\\
For (4) and (5), the same proof works if we replace \cite[Example 4.6]{MR1686947}
with the proof of \cite[Theorem 4.1]{MR1457742} and an algebraic proof \cite{MR1492525} of Angehrn and Siu's theorem \cite{MR1358978}
respectively.
\end{proof}

Other than \cite{MR1686947}, \cite{MR1457742}, \cite{MR1492525},
we could apply results in \cite{MR1911209}, \cite{MR4108218}, etc.
We note that
we assume lower bounds of $(B_{t}^{\dim Z}.Z)$ for \emph{all subvarieties} $Z$ in \autoref{lem_rho_dim2} contrary to \autoref{ques2}.

\section{Induction argument to abelian subvarieties}\label{section_induction}

Inspired by works \cite{MR2964474}, \cite{MR2833789}, \cite{MR4009173} on $(N_p)$ of polarized abelian varieties
and a generalization of Fujita's conjecture \cite{MR1492525},
the author asks the following question.

\begin{ques}[{\cite[Question 4.2]{MR3923760}}]\label{ques1}
Let $p$ be a non-negative integer,
$X$ be an abelian variety,
and $L$ be an ample line bundle on $X$.
Set $B=\frac{1}{p+2} L$.
Assume that $(B^{\dim Z} .Z) > (\dim Z)^{{\dim Z}}$ for any abelian subvariety $Z \subset X$ with $\dim Z \geq 1$.
Then does $L$ satisfy property $(N_p)$?
\end{ques}

From \autoref{thm_ Bpf_threshold} and \autoref{ques1},
Caucci asks the following question:

\begin{ques}[$=$ \autoref{ques2}]\label{ques2'}
Let $(X,l)$ be a polarized abelian variety.
Set $B_t=t l$ for a positive rational number $t > 0$.
Assume that $(B_t^{\dim Z} .Z) > (\dim Z)^{{\dim Z}}$ for any abelian subvariety $Z \subset X$ with $\dim Z \geq 1$.
Then does it hold that  $\beta(l) < t$?
\end{ques}

As we see in the previous section,
techniques and results to cut down lc centers are quite useful to study these questions.
However, we still have some difficulty to answer \autoref{ques2} 
even in dimension three. 
A generalization of Fujita's conjecture \cite[5.4 Conjecture]{MR1492525}
claims that an adjoint line bundle $K_X+L$ on a smooth projective variety $X$ is basepoint free if
$(L^{\dim X} ) > (\dim X)^{\dim X}$ and $(L^{\dim Z}.Z ) \geq (\dim X)^{\dim Z}$ for all subvarieties $Z \varsubsetneq X$.
Between \autoref{ques2} and \cite[5.4 Conjecture]{MR1492525},
there are the following two differences which cause some problems when we use \autoref{prop_local_enough} to answer \autoref{ques2}:
\begin{itemize}
\item[(a)] The lower bound $(\dim Z)^{{\dim Z}}$ in \autoref{ques2} is smaller than
the bound $ (\dim X)^{\dim Z}$ in \cite[5.4 Conjecture]{MR1492525}
for an abelian subvariety $Z \varsubsetneq X$.
\item[(b)] In \autoref{ques2}, we do not assume any lower bounds of $(B_t^{\dim Z}.Z) $ if $Z \varsubsetneq X$ is not an abelian subvariety.
\end{itemize}


For (a), 
we consider the following question.

\begin{ques}\label{conj1}
For an ample $\Q$-divisor $B$ on a $g$-dimensional abelian variety $X$ with $(B^g) > g^g$,
does there exist an effective $\Q$-divisor $F $ on  $X$
such that $B-F$ is nef and $\Nklt(X,F)$ contains an abelian subvariety of $X$
as an irreducible component?
\end{ques}

\autoref{conj1} is a weak version of 
a conjecture \cite[Conjecture 7.2]{lozovanu2018singular} by Lozovanu,
which requires $\Nklt(X,F)$ to be an abelian subvariety.
He proves his conjecture for threefolds under the assumption $(B^3) > 40$.

Although  Lozovanu  does not give a direct relation between his conjecture and \autoref{ques1},
he suggests that this kind of statement would be useful in induction type arguments.
Following his suggestion,
we  reduce \autoref{ques2} to \autoref{conj1}:

\begin{thm}\label{prop_conj1_ques3}
Let $g $ be a positive integer.
An affirmative answer to \autoref{conj1} for $\dim X \leq g$ implies that to \autoref{ques2} for $\dim X \leq g$.
\end{thm}


Before the proof of \autoref{prop_conj1_ques3}, we summarize relations between three questions as follows:
\begin{align*}
 \text{\autoref{conj1}$_{\leq g}$  $\rightsquigarrow$ \autoref{ques2}$_g$ $\rightsquigarrow$ \autoref{ques1}$_g$},
\end{align*}
where $Q \rightsquigarrow Q'$ means that an affirmative answer to $Q$ implies that to $Q'$,
and $Q_g$ or $ Q_{\leq g}$ means that we only consider the case when $\dim X =g $ or $\dim X \leq g$ respectively.
The first implication 
follows from \autoref{prop_conj1_ques3},
and the second one from \autoref{thm_ Bpf_threshold}. 
An important point is that the difference (a) does not cause any problem when we consider \autoref{conj1}
since we do not need to cut a lc center anymore if the center is an abelian subvariety. 
For \autoref{conj1}, the difference (b) still causes a problem when we cut down lc centers which are not abelian subvarieties.
We treat this problem in \autoref{sec_curve_hypersurface},\ref{sec_proof}.

\vspace{2mm}
To prove \autoref{prop_conj1_ques3},
we use the following
Poincar\'{e}'s reducibility theorem.
See \cite[Corollary 5.3.6]{MR2062673} for the proof.

\begin{thm}[Poincar\'{e}'s Reducibility Theorem]\label{lem_decomposition}
Let $(X,l)$ be a polarized abelian variety and $Y \subset X$ be an abelian subvariety.
Then there exists an abelian subvariety $Z \subset X$ 
such that the addition map $f : Y \times Z \rightarrow X $ is an isogeny with $f^* l = l|_Y \boxtimes l|_Z$.
\end{thm}

The following is a key proposition in the proof of \autoref{prop_conj1_ques3}.
This is a generalization of \autoref{prop_local_enough}.

\begin{prop}\label{prop_induction}
Let $(X,l)$ be a polarized abelian variety.
Assume that there exist a rational number $t >0$,
an effective $\Q$-divisor $F$ on $X$
and an abelian subvariety $ Y \subset X$ such that
\begin{itemize}
\item[(i)] $tl-F$ is ample,
\item[(ii)] 
$Y$ is an irreducible component of $\Nklt(X,F) $,
\item[(iii)] $\beta(l|_Y) < t$,
where we set $\beta(l|_Y) :=-\infty$ if $Y=\{o\}$.
\end{itemize}
Then $\beta(l) < t$.
\end{prop}

\begin{proof}
For a subset $S \subset X$ and a point $x \in X$,
the translations of $S$ by $x, -x$ are denoted by $S+x, S-x \subset X$ respectively.
Let $f : Y \times Z \rightarrow X$ be the isogeny in \autoref{lem_decomposition}.

By (ii),
there exists an open subset $ U \subset X$ such that
$\Nklt (X,F) \cap U = Y \cap U \neq \emptyset$.
Take a general $x \in X$. In particular, the finite set $Y \cap (Z-x) $ is contained in $U$.
Applying \autoref{lem_restriction_general_fiber} to $(X,F)$ and the quotient homomorphism $   X \rightarrow X/Z$,
we have
\begin{align*}
\Nklt (Z-x, F|_{Z-x}) \cap U  &=\Nklt (X,F) \cap (Z-x)  \cap U \\
&=(Y\cap U) \cap (Z-x) = Y \cap (Z-x).
\end{align*}

Set $W:=(Y+x ) \cap Z \subset U+x$ and $D_Z:= (F+x)|_Z \equiv F|_Z$.
Then
\begin{align}\label{eq_Nklt_D_Z}
\Nklt(Z,D_Z) \cap (U+x)  = \Nklt (Z, (F+x) |_{Z}) \cap (U+x)  =(Y+x) \cap Z =W.
\end{align}

Recall that $\beta(l) < t$ holds if and only if $ \cali_{x} \langle t  l \rangle$ is IT(0).
Since $f $ is an isogeny,
$ \cali_{x} \langle t  l \rangle$ is IT(0) if and only if $\cali_{f^{-1}(x)} \langle t f^* l \rangle $ is IT(0) by \autoref{lem_IT(0)} (3).

Since $W=(Y+x) \cap Z $ and $-Y=Y \subset X$, it holds that 
$f^{-1} (x) = \{ (x-w,w) \in Y \times Z \, | \, w \in W\} \subset Y \times W$.
Hence we have an exact sequence
\[
0 \rightarrow \cali_{Y \times W} \rightarrow \cali_{f^{-1}(x)} \rightarrow \cali_{f^{-1}(x)} /\cali_{Y \times W} \rightarrow 0
\]
on $Y \times Z$.
To prove that $\cali_{f^{-1}(x)} \langle t f^* l \rangle $ is IT(0), 
it suffices to show that both $\cali_{Y \times W}  \langle t f^* l \rangle $ and $(\cali_{f^{-1}(x)} /\cali_{Y \times W} ) \langle t f^* l \rangle $ are IT(0)
by \autoref{lem_IT(0)} (1).

\begin{claim}\label{clm_IT(0)_1}
$\cali_{Y \times W}  \langle t f^* l \rangle $ is IT(0).
\end{claim}

\begin{proof}
Since $\calo_Y \langle t  l|_Y  \rangle$ is IT(0) and 
\[\cali_{Y \times W}  \langle t f^* l \rangle = \calo_Y \boxtimes \cali_{W/Z}  \langle t ( l|_Y \boxtimes l |_Z) \rangle ,
\]
it suffices to show that $\cali_{W/Z} \langle t l |_Z \rangle$ is IT(0) by \autoref{lem_IT(0)} (2).
By (i), $t l |_Z  -D_Z  \equiv t l |_Z -F |_Z$ is ample.
By \ref{eq_Nklt_D_Z}, each point in $W$ is an isolated point of $\Nklt(Z,D_Z)$.
Hence $\cali_{W/Z} \langle t l |_Z \rangle$ is IT(0) by  \autoref{prop_local_enough'}.
\end{proof}

\begin{claim}\label{clm_IT(0)_2}
$(\cali_{f^{-1}(x)} /\cali_{Y \times W} ) \langle t f^* l \rangle $ is IT(0).
\end{claim}

\begin{proof}
Since $f^{-1} (x) = \{ (x-w,w) \in Y \times Z \, | \, w \in W\} \subset Y \times W$, we have
\[
\cali_{f^{-1}(x)} /\cali_{Y \times W}  = \bigoplus_{w \in W}  \cali_{(x-w,w)/ Y \times \{w\}} .
\]
Since $\beta (l|_Y) < t$ by (iii), $ \cali_{x-w/Y} \langle t  l|_Y \rangle$ is IT(0) for any $w \in W$.
Hence $ \cali_{(x-w,o)/ Y \times \{o\}}  \langle t f^* l \rangle  $ is IT(0) by \autoref{lem_IT(0)} (4).
Then $ \cali_{(x-w,w)/ Y \times \{w\}}  \langle t f^* l \rangle  $  is IT(0) and hence so is the direct sum $\cali_{f^{-1}(x)} /\cali_{Y \times W}  \langle t f^* l \rangle = \bigoplus_{w \in W}  \cali_{(x-w,w)/ Y \times \{w\}}  \langle t f^* l \rangle$.
\end{proof}

By Claims \ref{clm_IT(0)_1}, \ref{clm_IT(0)_2},
$\cali_{f^{-1}(x)} \langle t f^* l \rangle $ is IT(0).
Hence this proposition follows.
\end{proof}

Now we can prove \autoref{prop_conj1_ques3}.

\begin{proof}[Proof of \autoref{prop_conj1_ques3}]
We prove this theorem by the induction on $g$.
If $g=1$, 
there is nothing to prove since we already see that \autoref{ques2} has an affirmative answer for $\dim X=1$
by \autoref{lem_rho_dim2} (1).

Assume that this theorem holds for $g-1$.
Let $(X,l), B_t$ be as in \autoref{ques2} with $\dim X=g$.
What we need to show is that 
an affirmative answer to \autoref{conj1}$_{\leq g}$  implies $\beta(l) < t$.

Since $(B_t^{g}) > g^g$, we still have $(B^g) > g^g$ for $B=B_{t'}$ with $0 < t-t' \ll 1$.
If we assume an affirmative answer to \autoref{conj1}$_{g}$, 
there exists an effective $\Q$-divisor $F $ on $X$ such that $B-F$ is nef 
and $\Nklt(X,F)$ contains an abelian subvariety $Y \subset X$
as an irreducible component.
Since $tl -F=B_t -F =(B_t-B) +(B-F)$ is ample,
$F$ satisfies (i), (ii) in \autoref{prop_induction}.
If we assume an affirmative answer to \autoref{conj1}$_{\leq g-1}$, 
we have an affirmative answer to \autoref{ques2}$_{\leq g-1}$ by the induction hypothesis.
Hence $\beta(l|_Y) < t$ follows.
Thus (iii) in \autoref{prop_induction} is also satisfied
and $\beta(l) < t$ follows from \autoref{prop_induction}.
\end{proof}

\section{Two lemmas on curves and hypersurfaces on abelian varieties}\label{sec_curve_hypersurface}

In this section,
we show two lemmas on curves and hypersurfaces on abelian varieties.
In the proof of \autoref{main_thm},
we will apply these lemmas to minimal lc centers of $(X,D)$ for some $D$ on an abelian threefold $X$.

Throughout this section,
let $X, B$ be as in \autoref{conj1},
that is,
$X$ is an abelian $g$-fold and
$B$ is an ample $\Q$-divisor on $X$ with $(B^g) > g^g$.
For a subvariety $V \subset X$,
 $\langle V \rangle \subset X$ denotes the smallest abelian subvariety which contains a translate of $V$.
Let $\pi : X' \arw X $ be the blow-up at the origin $o \in X$
and $E \subset X'$ be the exceptional divisor.

\vspace{2mm}

The following lemma is a variant of \cite[Proposition 4.3]{lozovanu2018singular}.

\begin{lem}\label{lem_deg_curve}
For a curve $C \subset X$,
\begin{enumerate}
\item $(B.C) > g^2 / \sqrt[g]{g!}$ if $\langle C \rangle=X$.
\item If $S:=\langle C \rangle $ is a hypersurface,
$B-S$ is ample or $(B.C) > g(g-1) / \sqrt[g-1]{(g-1)!} $.
\end{enumerate}
\end{lem}

\begin{proof}
(1) By \cite[Proposition 4.1]{MR1294460}, we have $(B.C) \geq g \sqrt[g]{(B^g)/g!} > g^2 / \sqrt[g]{g!}$.\\
(2) If $(B^{g-1}.S) \leq g^{g-1}$, it holds that $(B^g) > g^g \geq g (B^{g-1}.S)$.
Hence $B-S$ is big by \cite[Theorem 2.2.15]{MR2095472}. 
Since $X$ is an abelian variety,
$B-S$ is ample.

If $(B^{g-1}.S) > g^{g-1}$,
we have $(B.C) \geq (g-1) \sqrt[g-1]{(B^{g-1}.S)/(g-1)!} >g(g-1) / \sqrt[g-1]{(g-1)!}  $ by \cite[Proposition 4.1]{MR1294460}.
\end{proof}

In the following lemma,
we combine the argument in \cite[Section 4]{MR1455517} with \cite{MR2135747}, \cite[Propositions 4.2, 4.4]{lozovanu2018singular}.
For the asymptotic multiplicity $\mult_{S'} ( \| \pi^* B - g E \|)$,
we refer the readers to \cite{MR2282673}, \cite{lozovanu2018singular}.

\begin{lem}\label{lem_S_is_abelian}
Let $S \subset X$ be a prime divisor such that $\mult_{S'} ( \| \pi^* B - g E \|) \geq 1$,
where $S' \subset X'$ is the strict transform of $S$.
Then $S$ is an abelian subvariety and $B-S$ is ample.
\end{lem}

\begin{proof}
By $\mult_{S'} ( \| \pi^* B - g E \|) \geq 1 $ and \cite[Lemma 3.3]{MR2282673}, we have $\vol(\pi^* B -gE)=\vol(\pi^* B -gE  -S')$.
Since $\pi^* B -gE$ is big by $(B^g) > g^g$,
so is $\pi^* B -gE  -S'$.
Hence $\pi_*(\pi^* B -gE  -S')=B-S$ is big as well.
Since $X$ is an abelian variety, 
the big divisor $B-S$ is ample.
Thus the rest is to show that $S$ is an abelian subvariety.

By the assumption $\mult_{S'} ( \| \pi^* B - g E \|) \geq 1 >0$,
$S' $ is contained in the 
restricted base locus $\mathbf{B}_{-}(\pi^* B - g E)$ by \cite[Proposition 2.8]{MR2282673}.
In particular,
$S'$ is contained in the augmented base locus $\mathbf{B}_{+}(\pi^* B - g E) $.
By \cite[Lemma 4.2]{lozovanu2018singular},
$o $ is contained in $S$ and if $S$ is smooth at $o$, $S$ is an abelian subvariety.
Hence it suffices to show that $\mult_o (S) =1$.

Set $\tau =\sup \{ t \geq 0 \, | \, \pi^* B -tE \text{ is pseudo-effective}\} >g$
and $ m=\mult_o(S) \geq 1$.
Define a function $m_{S'} : [0,\tau)  \rightarrow \R_{\geq 0}$ by $m_{S'} (t) = \mult_{S'} ( \| \pi^* B - t E \|) $.

By the definition of $m_{S'} (t) $ and \cite[Lemma 3.3]{MR2282673},
we have a bound of 
the restricted volume $ \vol_{X'|E} (\pi^* B - t E ) $ for $t \in (0,\tau)$ as
\begin{align*}
0 < \vol_{X'|E} (\pi^* B - t E ) &= \vol_{X'|E} (\pi^* B - t E - m_{S'} (t) S') \\
 & \leq \vol_{E} ((\pi^* B - t E - m_{S'} (t) S')|_E) \\
 &= \vol_{\P^{g-1}} \calo_{\P^{g-1}}(t-m m_{S'} (t)) 
\end{align*}
since $ S'|_{E} \sim \calo_E(m)$.
Hence we have $ t-m m_{S'} (t) >0 $ for $t \in (0,\tau)$ and
\begin{align}\label{eq_volume_bound}
(B^g) = g \int_{0}^{\tau}  \vol_{X'|E} (\pi^* B - t E ) dt \leq g  \int_{0}^{\tau}  (t-m m_{S'} (t))^{g-1} dt,
\end{align}
where the first equality follows from \cite[Corollary C]{MR2571958}.

By \cite[Lemma 3.4]{MR1393263} (see also \cite[Proposition 4.4]{lozovanu2018singular}),
we have $m_{S'} (t) - m_{S'} (g) \geq t -g  $ for $t \in [g,\tau)$.
Thus it holds that
  \begin{equation}\label{eq_upper_bound}
\begin{aligned}
 t - m m_{S'} (t)  &\leq t -m  (t-g+m_{S'} (g)) \\
 &= -(m-1)(t-g) + g- mm_{S'} (g)
\end{aligned}
  \end{equation}
for $t \in [g,\tau)$.
For simplicity, set $\alpha=g- mm_{S'} (g)$.
We have
\begin{align}\label{eq_bound_of_alpha}
0 <  \alpha=g- mm_{S'} (g) \leq g-m
\end{align}
by the assumption $m_{S'}(g) =\mult_{S'} ( \| \pi^* B - g E \|) \geq 1$.

Since $m_{S'} (t)$ is a convex function by \cite[Lemma 2.5]{lozovanu2018singular},
\[
\Lambda_1=\{ (t,y) \in \R^2 \,| \,  0 \leq t < \tau , 0 \leq y \leq   t - m m_{S'} (t) \}
\] 
is a convex set.
The set
\[
\Lambda_2=\{ (t,y) \in \R^2 \,| \,  t \geq g , y \geq  -(m-1)(t-g)  + \alpha \}
\]
is also convex.
Since $(g,\alpha) \in \Lambda_1 \cap \Lambda_2$ and  $\Lambda_1$ and $\Lambda_2$ do not intersect in the interiors 
by \ref{eq_upper_bound},
there exists $\lambda \geq m-1$ such that
\[
t - m m_{S'} (t)  \leq   -\lambda(t-g)  + \alpha
\]
for any $t \in [0,\tau)$.

Recall that we need to show $m=1$.
To prove this by contradiction, 
we assume $m \geq 2$.
Then we have $\lambda \geq m-1 \geq 1$.
Let $(a,a) $ be the intersection point of two lines $y=t, y=  - \lambda (t-g) + \alpha$,
and $(b,0)$ be the intersection point of two lines $y=0, y= - \lambda (t-g) + \alpha$,
that is,
\begin{align}\label{eq_a,b}
a = \frac{g\lambda + \alpha}{1+\lambda}, \quad b=  \frac{g\lambda + \alpha}{\lambda}.
\end{align}
We note that $ \tau \leq b$ follows since  $0 < t - m m_{S'} (t)  \leq   -\lambda(t-g)  + \alpha$ for any $t \in [ 0, \tau)$.

By \ref{eq_volume_bound} and the assumption $(B^g) > g^g$,
\begin{align*}
g^g < (B^g)  &\leq  g  \int_{0}^{\tau}  (t-m m_{S'} (t))^{g-1} dt \\
&=g \int_{0}^{a}  (t-m m_{S'} (t))^{g-1} dt  +g \int_{a}^{\tau}  (t-m m_{S'} (t))^{g-1} dt \\
&\leq g \int_{0}^{a}  t^{g-1} dt + g \int_{a}^{b}  ( -\lambda(t-g)  + \alpha)^{g-1} dt \\
&=a^g+ \frac{1}{\lambda}a^g =a^{g-1}b  \leq \left(  \frac{(g-1)a+b}{g}  \right)^g,
\end{align*}
where the last inequality follows from 
the fact that the geometric mean is bounded by the arithmetic mean.
Thus we have
\[
g  <  \frac{(g-1)a+b}{g} ,
\]
which is equivalent to 
\[
0 < ( \alpha +1 -g) g \lambda +\alpha 
\]
by \ref{eq_a,b}.
Since $\alpha \leq g-m$ by \ref{eq_bound_of_alpha} and $\lambda \geq m-1 \geq 1$,
it holds that
\begin{align*}
0 < ((g-m)+1 -g) g \lambda+ (g-m) &= -(m-1)g \lambda +g-m \\
&\leq -(m-1)^2 g +g-m.
\end{align*}
Hence $ (m-1)^2 < 1-\frac{m}{g} <1$, which contradicts $m \geq 2$. 
Thus we have $\mult_o(S)=m=1$.
\end{proof}

\section{Proof of \autoref{main_thm}}\label{sec_proof}

We answer \autoref{conj1} affirmatively for $\dim X \leq 3$:

\begin{thm}\label{thm_abelian_center_dim3}
\autoref{conj1} has an affirmative answer for $\dim X \leq 3$,
that is,
for an ample $\Q$-divisor $B$ on a $g$-dimensional abelian variety $X$ with $(B^g) > g^g$ and $g \leq 3$,
there exists an effective $\Q$-divisor $F$ such that $B-F$ is nef and $\Nklt(X,F)$ contains an abelian subvariety as an irreducible component.
\end{thm}

\begin{proof}
If $\dim X=1$, we may take $F=o $. 
If $\dim X =2$, under the assumption $(B^2) >4$,
the proof of \cite[Proposition 3.1]{MR3923760} constructs an effective $\Q$-divisor $D$ 
such that $B-D$ is ample, $(X,D)$ is lc, and a minimal lc center of $(X,D)$ is a point or an elliptic curve.
Perturbing $D$ slightly by \autoref{lem_perturbation},
we obtain $F$ as in the statement of this theorem.

Assume $\dim X=3$ and let $B$ be an ample $\Q$-divisor with $(B^3) > 27$.
Let $\pi : X' \arw X $ be the blow-up at $o \in X$ and $E \subset X'$ be the exceptional divisor.
Since $(B^3) > 27$, 
we can take a rational number $\delta >0 $ such that $(B^3) > (3+\delta)^3$.
In particular,
$\pi^* B - (3+\delta) E$ is big. 
Take a general effective $\Q$-divisor $D' \equiv \pi^* B - (3+\delta) E$. 
Set $D=\pi_*(D') \equiv B$ and 
\begin{align}\label{eq_c}
c:=\lct_o(X, D) \leq \frac{3}{3+\delta} <1.
\end{align}

If the minimal lc center of $(X,cD)$ at $o$ is $\{o\}$,
a small perturbation $F$ of $cD$ satisfies the condition in the statement of this theorem.
Thus we may assume that the minimal lc center of $(X,cD)$ at $o$ is a curve or a surface.

\vspace{2mm}
\noindent
{\bf Case} $1$: The minimal lc center is a curve $C$.
\begin{itemize}
\item[Case 1-1.]  $\dim \langle C \rangle = 3$: 
In this case, we have $ (B.C) > 9 / \sqrt[3]{6} > 3$ by \autoref{lem_deg_curve} (1).
Thus we can apply \autoref{lem_cut_center_curve}
to the above $X,B,D$ by \autoref{rem_cut_center_curve},
and hence there exists an effective $\Q$-divisor 
$D_1 \equiv c_1 B$ for some $c_1 < 1$ such that
$(X,D_1)$ is lc but not klt at $o$ and the minimal lc center is $\{o\}$.
Perturbing $D_1$ slightly, 
we have an effective $\Q$-divisor $F$ on $X$ such that 
$B-F$ is ample and $\{o\}$ is an isolated point of $\Nklt(X,F)$.
\item[Case 1-2.]  $\dim \langle C \rangle = 2 $: In this case, $S= \langle C \rangle $ is an abelian surface.
By \autoref{lem_deg_curve} (2), $B-S$ is ample or $(B.C) > 3 \sqrt{2}$.

If $B-S$ is ample,
$F=S$ satisfies the condition in the statement of this theorem.

If $(B.C) > 3 \sqrt{2} >3$, 
we have  $F$ such that 
$B-F$ is ample and $\{o\}$ is an isolated point of $\Nklt(X,F)$ as in Case 1-1.
\item[Case 1-3.]  $\dim \langle C \rangle = 1 $: In this case, $C$ is an elliptic curve.
A small perturbation $F$ of $cD$ satisfies the condition in the statement of this theorem. 
\end{itemize}

\vspace{2mm}
\noindent
{\bf Case} $2$: The minimal lc center is a surface $S$.

\begin{claim}\label{claim_mult}
Let $S' \subset X'$ be the strict transform of $S$.
Then  $\mult_{S'} ( \| \pi^* (\frac{3}{3 + \delta }B) - 3 E \|) \geq 1$.
\end{claim}

\begin{proof}[Proof of \autoref{claim_mult}]
Assume $\mult_{S'} ( \| \pi^* (\frac{3}{3 + \delta }B) - 3 E \|) < 1$, that is, $\mult_{S'} ( \| \pi^* B - (3+\delta) E \|) <  (3+\delta)/3$.
Then we have $ \mult_{S'} (D') < (3+\delta)/3$ since we take general $D' \equiv \pi^* B - (3+\delta) E$.
On the other hand, since $S$ is a lc center of $(X,cD)$, 
we have $\mult_{S'} (D') =\mult_S (D) =c^{-1} \geq (3+\delta)/3$ by \ref{eq_c}, which is a contradiction.
\end{proof}

Since $\mult_{S'} ( \| \pi^* (\frac{3}{3 + \delta }B) - 3 E \|) \geq 1$ and $((\frac{3}{3 + \delta} B)^3) = (\frac{3}{3 + \delta})^3 (B^3) > 27$,
$S$ is an abelian surface and $\frac{3}{3 + \delta} B-S$ is ample by \autoref{lem_S_is_abelian}.
Hence $F=S$ satisfies the condition in the statement of this theorem.
\end{proof}

\begin{proof}[Proof of \autoref{main_thm}]
\autoref{main_thm} is nothing but an affirmative answer to \autoref{ques2} for $\dim X=3$.
Hence it follows from Theorems \ref{prop_conj1_ques3}, \ref{thm_abelian_center_dim3}.
\end{proof}

\begin{rem}\label{rem_connected component}
We note that \autoref{prop_local_enough} is not enough to prove \autoref{main_thm}
even if we use lemmas in \autoref{sec_curve_hypersurface}.
The reason is as follows:

Let $X,l,B_t=tl$ be as in \autoref{main_thm} and 
take $D\equiv B_t, c=\lct_o(D)$ 
as in the proof of \autoref{thm_abelian_center_dim3}.
Assume that the minimal lc center of $(X,cD) $ at $o$ is an elliptic curve $C$ as in Case 1-3.
Then we cannot cut down $C$ since we only assume $(B_t.C) >1$
and cannot apply \autoref{lem_cut_center_curve}.
This is the case when we essentially need \autoref{prop_conj1_ques3} or \autoref{prop_induction}.

If we perturb $cD$, 
we have $\calj(X,cD) =\cali_C $ in a neighborhood of the origin. 
It might be worth noting that
we obtain $\beta(l) < t$ if $\calj(X,cD) =\cali_C $ holds on \emph{whole} $X$ as follows:
If $\calj(X,cD) =\cali_C $, then $\cali_C \langle tl \rangle =\calj(X,cD) \langle tl \rangle$ is IT(0)
by the Nadel vanishing theorem.
Since $ \deg \calo_C(-o) \langle tl \rangle =(B_t.C) -1>0$ and $C$ is an elliptic curve,
$\cali_{o/C} \langle tl \rangle =\calo_C(-o) \langle tl \rangle$ is IT(0) as well.
Thus $\cali_o \langle tl \rangle$ is IT(0) and hence $\beta(l)< t$ follows from the exact sequence
\begin{align}\label{eq_short_exact_C}
0 \rightarrow \cali_C \rightarrow \cali_o \rightarrow \cali_{o/C} \rightarrow 0
\end{align}
and \autoref{lem_beta_IT(0)} (1).
However,
we do not know how to control the singularity of $(X,cD)$ outside the origin and hence 
do not know whether $\calj(X,cD) =\cali_C $ holds or not on whole $X$.
Thus we cannot apply the above argument to prove \autoref{main_thm},
at least without further argument about the singularities of $(X,cD)$.

The readers might ask if we use $c'=\lct(D) := \max\{ s \geq 0 \, | \, (X,s D) \text{ is log canonical} \}$ instead of $c=\lct_o(D)$.
Then $\calj(X,c'D) =\cali_{C'} $ holds on whole $X$ for a minimal lc center $C'$ of $(X,c'D)$ after perturbing $c'D$
and hence the above argument to use \ref{eq_short_exact_C} works if $C'$ is an elliptic curve.
However, if $C'$ is a curve of general type,
it is not clear whether $\cali_{o/C'} \langle tl \rangle $ is IT(0) or not.
Furthermore, we cannot apply \autoref{lem_cut_center_curve} to cut down $C'$
since $C'$ might not contain the origin and hence 
there might not exist a point $x \in C'$ with $\mult_x (D) > 3$.
Therefore, we are not sure how to treat the case
when $C'$ is a curve of general type 
if we consider $\lct(D) $ instead of $\lct_o(D)$.
\end{rem}

\begin{proof}[Proof of \autoref{cor_main}]
By \autoref{main_thm}, we have $\beta(l) < 1/(p+2)$.
Hence $L$ satisfies property $(N_p)$ 
by \autoref{thm_ Bpf_threshold}.
The vanishing of $K_{p,q}(X,L;dL)$ follows from \autoref{prop_vanishing_K}.
\end{proof}

\begin{proof}[Proof of \autoref{cor_ Bpf}]
If $(L^2.S) \geq 6 $ and $(L.C) \geq 2$ for any abelian surface $S \subset X$ and any elliptic curve $C \subset X$, then
$\beta(l) <1$ follows from \autoref{main_thm} and the assumption $(L^3) \geq 30 > 27$. Hence $L$ is basepoint free by \autoref{thm_ Bpf_threshold} (1).

On the other hand,
if $L$ is basepoint free,
so are $L|_S$ and $L|_C$ for any abelian surface $S$ and any elliptic curve $C $ in $ X$.
Since $L|_S$ and $L|_C$ are ample,
$h^0( S,L|_S) =  (L|_S^{2})/2 \geq 3$ and $h^0( C,L|_C) = \deg (L|_C) \geq 2$.
Thus we have $(L^2.S) \geq 6 , (L.C) \geq 2$.
\end{proof}

If $h^0(L) =(L^3)/3! \leq 3$,
$L$ is not basepoint free since $L$ is ample.
Hence if $h^0(L) \neq 4$,
the basepoint freeness of 
an ample line bundle $L$ on an abelian threefold
is completely determined by the numerical data
$(L^{\dim Z}.Z )$ for all abelian subvarieties $Z \subset X$,
including the data for $Z=X$.

However,
the basepoint freeness is not determined by such data
if $h^0(L) = 4$.
For example,
let $(X_1,l_1)$ and $(X_2,l_2) $ be very general polarized abelian threefolds of type $(1,1,4)$ and $(1,2,2)$ respectively.
Let $L_i$ be a line bundle representing $l_i$ for each $i=1,2$.
Then $(L_1^{3}) =(L_2^{3})= 24$.
Furthermore,
$X_1,X_2$ are simple since $(X_1,l_1)$ and $(X_2,l_2) $ are very general.
Hence the numerical data $(L^{\dim Z}.Z )$ for all abelian subvarieties $Z \subset X$
coincide for $X_1$ and $X_2$.
On the other hand,
$L_1$ is basepoint free by \cite[Proposition 2]{MR1299059},
but $L_2$ is not by \cite[Corollary 2.6]{MR1360615}, \cite[Remark 15.1.4]{MR2062673}.

In particular,
the assumption $h^0(L) \geq 5$ is sharp,
that is,
the statement of \autoref{cor_ Bpf} does not hold under the weaker assumption $h^0(L) \geq 4$.

\bibliographystyle{amsalpha}
\newcommand{\etalchar}[1]{$^{#1}$}
\providecommand{\bysame}{\leavevmode\hbox to3em{\hrulefill}\thinspace}
\providecommand{\MR}{\relax\ifhmode\unskip\space\fi MR }
\providecommand{\MRhref}[2]{%
  \href{http://www.ams.org/mathscinet-getitem?mr=#1}{#2}
}
\providecommand{\href}[2]{#2}

\end{document}